\documentclass{article}

\usepackage[english]{babel}
\usepackage[utf8x]{inputenc}
\usepackage[T1]{fontenc}
\usepackage{lmodern}
\usepackage{latexsym}
\usepackage{amsthm}
\usepackage{a4}
\usepackage{amsfonts}
\usepackage{bm}
\usepackage{amssymb}
\usepackage{subfigure}
\usepackage{eucal}
\usepackage{amscd}
\usepackage{amsmath,mathtools}
\usepackage{array}
\usepackage{float}
\usepackage{setspace}
\usepackage{listings}
\usepackage{color} 
\usepackage{graphicx}
\usepackage{xifthen}
\usepackage{ragged2e}
\usepackage{yfonts}
\usepackage{esint}
\usepackage{comment}
\usepackage[dvipsnames]{xcolor}

\usepackage[colorinlistoftodos]{todonotes}
\setuptodonotes{color=green!40, fancyline, size=small}

\usepackage[top=2.8cm, bottom=2.5cm,
        left=2.1cm, right=1.7cm]{geometry}
\usepackage[open, openlevel=1]{bookmark}

\DeclarePairedDelimiter{\norma}{\lVert}{\rVert}
\DeclareMathOperator{\spn}{span}

\usepackage{csquotes}
\usepackage[backend=bibtex,sorting=none]{biblatex}
\usepackage{amsfonts,amssymb}
\usepackage{nccmath}
\usepackage{comment}

\makeatletter
\renewcommand*\env@matrix[1][\arraystretch]{%
  \edef\arraystretch{#1}%
  \hskip -\arraycolsep
  \let\@ifnextchar\new@ifnextchar
  \array{*\c@MaxMatrixCols c}}
\makeatother

\newcommand{\R}{\mathbb{R}}
\newcommand{\N}{\mathbb{N}}

\newcommand{\K}{\mathbb{K}}
\newcommand{\D}{\mathbb{D}}
\newcommand{\DD}{\mathcal{D}}
\newcommand{\A}{\mathcal{A}}
\newcommand{\eye}{\mathbb{I}}
\newcommand{\B}{\mathcal{B}}

\newcommand{\weakstar}{\overset{\ast}{\rightharpoonup}}
\newcommand{\weak}{\rightharpoonup}

\newcommand{\e}{\varepsilon}

\newcommand{\dx}{\,\text{d}x}

\newcommand{\dt}{\,\text{d}t}
\newcommand{\ds}{\,\text{d}s}

\newcommand{\rsym}{\R^{2\times2}_{sym}}

\newcommand{\F}{\mathcal{F}}
\newcommand{\U}{\mathcal{U}}

\newcommand{\y}{\bar{y}}
\newcommand{\RR}{\bar{R}}

\newcommand{\begstat}{Let $\bar{y}\in\R^3$ and $\bar{R}\in SO(3)$.}

\newcommand{\qstar}{Q^{**}}
\newcommand{\qstart}{Q^{**}_t}

\newcommand{\mob}{M\"obius}

\newcommand{\sign}{\text{sign}}

\newcommand{\innerproduct}[2]{\langle #1, #2 \rangle}

\newtheorem{theorem}{Theorem}[section]

\newtheorem{lemma}{Lemma}[section]

\newtheorem{remark}{Remark}

\newtheoremstyle{exercise}
  {\topsep}   
  {\topsep}   
  {}          
  {}          
  {\bfseries} 
  {.}         
  {.5em}      
  {}          

\theoremstyle{exercise}

\newtheorem{definition}{Definition}[section]

\title{ON THE SADOWSKY FUNCTIONAL FOR ANISOTROPIC RIBBONS}
\author{
{\sc Giovanni Savar\'e}
}
\date{
}

\begin{document}
\maketitle

\begin{abstract}
\noindent The equilibrium shape of a thin, elastic, inextensible ribbon minimizes its bending energy.
It has been shown that, as the width of the ribbon tends to zero, the bending energy $\Gamma$-converges to the so called Sadowsky functional.
In this paper we consider geometrically frustrated anisotropic ribbons with a possibly curved reference configuration. We prove that the $\Gamma$-convergence remains valid under prescribed affine boundary conditions, including, in particular, those satisfied by a Möbius strip.
\end{abstract}

\section{Introduction}
Thin inextensible ribbons can be modeled as two-dimensional surfaces that are locally isometric to a flat reference configuration. 
The equilibrium shape of a ribbon is a minimiser of the deformation energy of the strip, which is entirely due to bending, as it was first pointed out by Sadowsky in \cite{sadovsky_1}. In \cite{sadovsky_2} Sadowsky raised the question of finding the characteristic equilibrium shape of a \mob{} band, by minimising the bending energy over the class of developable surfaces, subject to the boundary conditions of a \mob{} strip. Since a \mob{} strip can be obtained by glueing together the two short sides of a rectangle after a half twist, these boundary conditions can be expressed as affine constraints on the short sides of a rectangular reference configuration.
If we denote by $S_\e = (0,\ell)\times (-\e/2, \e/2)$ the reference configuration of the ribbon, where $\ell>0$ and $\e>0$ are the length and the width of the strip respectively, then, according to Sadowsky, the bending energy of a deformation $u:S_\e\to\R^3$ is given (up to a multiplicative constant) by

\begin{equation}\label{eq:en_functional}
    E_\e(u) = \frac{1}{\e}\int_{S_\e}Q(\Pi_u(x)-\Pi^{nat}_\e)\dx,
\end{equation}
where $\Pi_u$ is the second fundamental form of $u$, $\Pi^{nat}_\e\in L^2\bigl((0, \ell);\rsym\bigl)$ represents either the second fundamental form of a certain deformation, or a target curvature tensor, and $Q:\rsym\to\R_{\ge0}$ is the bending energy density, which we will always assume to be a positive definite quadratic form. The term $\Pi_\e^{nat}$ allows to include in the analysis geometrically frustrated ribbons or ribbons whose natural configuration is not necessarily flat. Finally, the energy density $Q$ is a general quadratic form accounting for a possible elastic anisotropy of the material.

The isometry constraint on the ribbons reads as $(\nabla u)^T\cdot(\nabla u) = \eye.$ This is a nonlinear relation that makes highly nontrivial to perform variations of $E_\e$ and to compute the Euler-Lagrange equations.
In order to simplify this problem, Sadowsky considered a dimension reduction approach. In \cite{sadovsky_1} he argued that for infinitesimally narrow bands (that is, as $\e\to 0$), the energy $E_\e$ reduces to 
\begin{equation}\label{eq:int_1}
    \hat{E}(y) = \int_{0}^{\ell}\frac{(k^2+\tau^2)^2}{k^2}\ds,
\end{equation}
where $k$ and $\tau$ are, respectively, the curvature and the torsion of the centerline $y$ of the strip, and $s$ is its arc-length. However, Sadowsky did not provide any proof of this result, and he tacitly assumed the centerline of the band to have strictly positive curvature.

Wunderlich in \cite{wunderlinch} gave a formal justification of Sadowsky's claim, showing that, under the assumption of nonvanishing curvature, the energy $\hat{E}$ is in fact the pointwise limit of $E_\e$, as $\e\to 0.$

An other justification of (\ref{eq:int_1}) was proposed by Kirby and Fried in \cite{kirby}. Again under the assumption that the curvature of the centerline is bounded away from zero, they proved that $\hat{E}$ is the $\Gamma$-limit of $E_\e$, as $\e\to 0$, with respect to the strong topology of $W^{3, p}(0, \ell),$ with $p\ge 1$. However, this topology is not the natural one, since compactness of sequences of minimisers fails, and thus the Fundamental Theorem of $\Gamma-$convergence cannot be applied.

In \cite{mora3}, \cite{mora2} and \cite{mora1} Freddi, Hornung, Mora and Paroni gave a rigorous justification of Sadowsky's claim by $\Gamma$-convergence in the natural topology, and showed that the expression of the limit functional (\ref{eq:int_1}) needs to be corrected for small curvatures. 
In \cite{mora2}, Freddi, Hornung, Mora and Paroni studied the $\Gamma-$convergence of the functionals $E_\e$ (after a proper rescaling) defined on the space of isometries belonging to $W^{2, 2}(S_\e;\R^3)$. They considered a general positive definite quadratic form $Q$, and a general flat reference configuration, but they did not enforce any boundary conditions. In this context the energies $E_\e$, properly rescaled, $\Gamma-$converge to the functional 
\begin{equation}\label{eq:limit_functional_first}
    J(y, d_1, d_2, d_3) \coloneqq \int_I\bar{Q}(x_1, d_1'\cdot d_3, d_2'\cdot d_3)\dx_1
\end{equation}
(see (\ref{def:limit_functional})),
defined on the space $\A_0$. The energy density $\bar{Q}$ in (\ref{eq:limit_functional_first}) is defined in (\ref{eq:limit_energy_density}). The elements of the space $\A_0$ are framed curves $(y, (d_1, d_2, d_3))$, where $y$ represents the centerline of the limiting strip, $d_1 = y'$ is the tangent to the centerline, $d_2$ represents the transversal orientation of the strip, and $d_3$ is the normal to the strip. The frame $(d_1, d_2, d_3)$ may be not orthonormal, and is related to the geometry of the reference configuration, see (\ref{eq:limit_space}).

The main goal of this work is to prove that this last more general $\Gamma$-convergence result is stable under prescribed affine boundary conditions. In particular, we require the ``short left side'' of the strip to be clamped, and the ``right short side'' to be isometrically deformed, see Subsection \ref{subsec:boundary_conditions} for a more precise statement. These boundary conditions are the same enforced in \cite{mora1}, where the $\Gamma$-convergence has been studied in the case $Q(M) = |M|^2$ and under the assumption of a rectangular stress-free configuration.

While the liminf inequality and the compactness Lemma easily follow from the corresponding results in \cite{mora2}, the proof of the limsup inequality requires some highly non trivial changes, although the strategy is the same as in \cite{mora1}. The original contribution of this work is contained in Section \ref{sec:refinement}, and is mostly given by Theorems \ref{teo:perturbation} and \ref{teo:approx1}, which give a refinement of the relaxation result \cite{mora2}[Proposition 9]. This result is a fundamental building block for the recovery sequence and is the following. Let
\begin{equation*}
    Q^0(M) = \begin{cases}
        Q(M)&\text{ if }\det(M) = 0,\\
        +\infty & \text{otherwise},
    \end{cases}\quad\text{ and }\quad\hat{\F}:L^2(I;\rsym)\to[0, +\infty],\;\;\hat{\F}(M) = \int_IQ^0(M(x))\dx.
\end{equation*}
In \cite[Proposition 9]{mora2} the authors proved that the lower semicontinuous envelope of $\hat{\F}$ with respect to the weak topology of $L^2(I;\rsym)$ is the functional 
\[\F:L^2(I,\rsym)\to[0, +\infty],\;\;\F(M)\coloneqq
\int_0^\ell \Bigl(Q(M(x))+\alpha^+(\det M(x))^++\alpha^-(\det M(x))^-\Bigl)\dx,\]
where the constants $\alpha^+,\alpha^-$ are defined in \eqref{alpha_p} and \eqref{alpha_m}.
In Theorems \ref{teo:perturbation} and \ref{teo:approx1} we give a refinement of this result. Moreover, given a piecewise constant function $M\in L^2(I;\rsym)$ we explicitly construct a sequence $M^n\subset L^2(I;\rsym)$ weakly converging to $M$ in $L^2$ and such that $\hat{\F}(M^n)\to\F(M)$. This construction requires a completely new strategy with respect to the isotropic case in \cite{mora1}.
For more details we refer to the discussion after Remark \ref{rem:uniform_convergence}.

Here we provide an outline of the construction of the recovery sequence for Theorem \ref{teo:gamma_convergence}.
Suppose for simplicity that $\Pi_\e^{nat} = 0$ for every $\e$, and that the reference configuration of the \mob{} strip is the rectangle $(0, \ell)\times(-\e/2, \e/2)$ (see Subsection \ref{subsec:reference_configuration} for the general case). Let $(y, R)\in\A_0$.
Roughly speaking, the strategy consists in first constructing a sequence of framed curves $(y_\e, R_\e)$ converging to $(y, R)$, and then extending these curves to isometries $u_\e$ with the prescribed boundary conditions. Each of these isometries is defined as a ruled surface with centerline $y_\e$. A necessary condition for $u_\e$ to be an isometry is that its second fundamental form has zero determinant (according to Gauss' Theorema Egregium). The approximating sequence $(y_\e, R_\e)$ is chosen in such a way that this constraint is satisfied on the centerline.

Setting $\mu(t)\coloneqq d'_1(t)\cdot d_3(t)$ and $\tau(t) \coloneqq d_2'(t)\cdot d_3(t)$, the limiting energy rewrites as
\(J(y, R) = \F(M),\) where $\displaystyle M = \begin{pmatrix}
    \mu&\tau\\\tau&\gamma
\end{pmatrix},$
for a suitable function $\gamma\in L^2(I).$

By the relaxation result \cite[Proposition 9]{mora2} there is a sequence $(M^\e)_\e$ weakly converging to $M$ in $L^2$ such that $\det(M^\e) = 0$ and $\F(M^\e)\to\F(M)$, as $\e\to0$. For every $\e>0$ we set $\mu_\e\coloneqq M^\e_{11}$ and $\tau_\e\coloneqq M^\e_{12}$. A frame $(y_\e, R_\e)$ satisfying $\mu_\e = (d^\e_1)'\cdot d_3^\e,\,\tau_\e = (d^\e_2)'\cdot d^\e_3$ and $(d^\e_1)'\cdot d^\e_2 = 0$ can now be obtained by solving a suitable Cauchy problem (see also \eqref{eq:cauchy_problem}).
Using that $\det(M^\e(x)) = 0$ a.e. in $(0, \ell)$ by construction, \cite[Proposition 13]{mora2} shows that $M^\e$ can be interpreted as the second fundamental form of an isometry $u_\e$ restricted to the midline of the strip. Since $\F(M^\e)\to\F(M)$, the convergence $E_\e(u_\e)\to \F(M) = J(y, R)$ easily follows.

However the isometries $u_\e$ do not satisfy in general the required boundary conditions. To solve this issue the sequence $(M^\e)_\e$ has to be chosen so that $(y_\e, R_\e)$ satisfies the same boundary conditions as the elements of $\A_0$.
In order to do so, we follow the approach of \cite{mora1}: using some results due to Hornung \cite{hornung}, we can find perturbations $\tilde{\tau}_\e, \tilde{\mu}_\e$ of $\tau_\e$ and $\mu_\e$ respectively, such that the framed curves $(\tilde{y}_\e, \tilde{R}_\e)$ associated with $\tilde{\mu}_\e$ and $\tilde{\tau}_\e$ weakly converge to $(y, R)$, and have the same boundary conditions. The energy associated with the deformed frame is $\F(\tilde{M}^\e)$, with
\[\tilde{M}^\e\coloneqq \tilde{\mu}_\e e_1\otimes e_1+\tilde{\tau}_\e(e_1\otimes e_2+e_2\otimes e_1)+\tilde{\gamma}_\e e_2\otimes e_2,
\]
where $\tilde{\gamma}_\e$
is uniquely determined through the zero determinant constraint $\tilde{\mu}_\e\cdot\tilde{\gamma}_\e = \tilde{\tau}_\e^2$ whenever $\tilde{\mu}_\e\neq 0$. We note that, if $\tilde{\mu}_\e$ is close to zero, we may lose control of $\tilde{\gamma}_\e$. Therefore, in order to achieve the convergence ${\F}(\tilde{M}^\e)\to\F(M),$ it is crucial to bound $\mu_\e$ (and $\tilde{\mu}_\e$) away from zero, so that also $\tilde{\gamma}_\e$ remains bounded. Addressing this issue requires a refinement of the relaxation result of \cite{mora2}. Addressing this issue is the original contribution of this work, and is the main topic of Section \ref{sec:refinement} (see also the discussion after Remark \ref{rem:uniform_convergence} for a detailed outline of Section \ref{sec:refinement}).

\section{Setting of the problem}

\subsection{Reference configuration}\label{subsec:reference_configuration}
Given $\ell>0$ we denote $I= (0, \ell)$ and $\Omega = I\times(-1/2, 1/2).$ For $\e>0$, we define $\rho_\e(x_1, x_2)\coloneqq (x_1, \e x_2)$ and the rectangular region $\Omega_\e = \rho_\e(\Omega) = I\times(-\e/2, \e/2)$. The reference configuration of the inextensible elastic strip is $S_\e\coloneqq \chi(\Omega_\e)\subseteq \R^2$, where $\chi:\R^2\to\R^2$ is an orientation preserving injective map of class $C^2$ such that 
\begin{equation}\label{eq:100}
    |\partial_1\chi(x_1, 0)| = 1\;\forall\,x_1\in\R.
\end{equation}
In the following we call $B(t)\coloneqq\chi(t, 0)$ the midline of the reference configuration, which is parametrised by arc length because of (\ref{eq:100}).
By composing $\chi$ with an orientation preserving isometry, we may further assume that
\(\chi(0, 0) = (0, 0)\) and \(\partial_1\chi(0, 0) = e_1.\)

\subsection{The energy density}
We describe the spatial configuration of the thin inextensible elastic ribbon with an isometry $u\in W^{2, 2}(S_\e; \R^3)$. The inextensibility constraint reads
\[(\nabla u)^T(\nabla u) = \eye,\]
where $\eye$ is the $2\times 2$ identity matrix.
We call $\nu_u\coloneqq\partial_1u\wedge\partial_2u$ the unit normal to the deformed configuration, and we denote by $\Pi_u:S_\e\to\rsym$ the second fundamental form of the deformation, defined by $(\Pi_u)_{\alpha\beta}\coloneqq\nu_u\cdot\partial_\alpha\partial_\beta u$. Since $u$ is an isometry, by Gauss' Theorema Egregium the Gaussian curvature of the strip is zero, that is, $\det(\Pi_u) = 0$ in $S_\e$.

The elastic energy of a spatial configuration $u$ of the ribbon is given by
\begin{equation}\label{eq:en_density}
    E_\e(u)\coloneqq \frac{1}{\e}\int_{S_\e} Q\bigl(\Pi_u(x)-\Pi_\e^{nat}(x)\bigl)\dx= \frac{1}{\e}\int_{S_\e} \K\bigl(\Pi_u(x)-\Pi_\e^{nat}(x)\bigl)\cdot \bigl(\Pi_u(x)-\Pi_\e^{nat}(x)\bigl)\dx.
\end{equation}
The energy density, $Q:\rsym\to\R$ is a positive definite quadratic form which accounts for the possible material's anisotropy, and $\K$ is the associated linear map from $\rsym$ into itself which, in particular, verifies
\[\K A\cdot B = \K B\cdot A, \text{ and }Q(A) = \K A\cdot A\ge c|A|^2 \text{ for every }A, B\in\rsym,\,\text{ for some }c>0.\]
The quantity $\Pi_\e^{nat}\in L^2(S_\e;\rsym)$ in (\ref{eq:en_density}) is a symmetric tensor field, representing the second fundamental form of a natural configuration, or a target curvature tensor field not necessarily corresponding to a configuration.

\subsection{Boundary conditions}\label{subsec:boundary_conditions}
In the following, we refer to the left and right short sides of $S_\e$ as the images of the curves $t\mapsto\chi(0, t)$ and $t\mapsto \chi(\ell, t)$, respectively, which are defined for $t\in(-\e/2, \e/2)$. Similarly, the images of the left and right short sides under a deformation $u$ will be called the left and right short sides of the strip $u(S_\e)$.

Let $\bar{y}\in\R^3$ and $\bar{R}^T = (\bar{d}_1|\bar{d}_2|\bar{d}_3)\in SO(3)$. We impose the following boundary conditions on the short sides of the strip $u(S_\e)$:
\begin{equation}\label{eq:bc}
\begin{cases}
    &u(0, 0) = (0, 0, 0) \text{ and } u(\chi(\ell, 0)) = \bar{y},\\
    &\nabla u(\chi(0, x_2)) = (e_1|e_2) \text{ for every } x_2\in (-\e/2, \e/2),\\
    &\nabla u(\chi(\ell, x_2)) = (\bar{d}_1|\bar{d}_2) \text{ for every } x_2\in (-\e/2, \e/2),
\end{cases}
\end{equation}
where the equalities are meant in the sense of the boundary traces.
These conditions amount to requiring that the left short side of the strip is clamped, and the right short side is isometrically deformed.

From now on $\A_\e$ will denote the space of admissible deformations, that is,
\begin{equation}
    \A_\e\coloneqq\{u\in W^{2, 2}(S_\e;\R^3):\:(\nabla u)^T(\nabla u) = \eye \text{ a.e. in }S_\e \text{ and }u\text{ satisfies }(\ref{eq:bc})\}.
\end{equation}
We point out that the set $\A_\e$ may be empty for some specific values of $\bar{y}$ and $\bar{R}$. For instance, from (\ref{eq:100}) it follows that necessarily \\$|\bar{y}| = |u(\chi(\ell, 0))-u(\chi(0, 0))| \le \ell,$ see also \cite{mora1} for some characterizations of $\A_\e$ in the case $\chi(x_1, x_2) = (x_1, x_2)$. We will discuss again this problem in Subsection \ref{subsec:compactness_and_limit_space}; see also Remark \ref{rem:non_empty_space}.

\subsection{The rescaled energy density}
Let $\chi_\e:\Omega\to S_\e$ be defined by
\(\chi_\e\coloneqq\chi\circ\rho_\e\). We can associate to a given deformation $u:S_\e\to\R^3$ the rescaled deformation $y:\Omega\to\R^3$ defined as
\[y\coloneqq u\circ\chi_\e.\]
We define $D^\e\coloneqq(\nabla\chi)\circ\rho_\e$, and for $\alpha=1, 2$ we set $D^\e_\alpha\coloneqq ((\nabla\chi)\circ\rho_\e) e_\alpha$ the directors defining the geometry of $S_\e$. We point out for future reference that
\begin{equation}\label{eq:4}
    D^\e\to\nabla\chi(\cdot, 0)\eqqcolon D
\end{equation}
uniformly and that there exists a constant $c>0$ such that
\begin{equation}\label{eq:5}
    c\le\det D^\e(x)\le1/c,\text{ and }c\le|D^\e(x)|\le 1/c\text{ for every }x\in\Omega.
\end{equation}
It is immediate to see that
\begin{equation}
\label{rescaled_boundary_cond_1}
    \begin{cases}
    \begin{aligned}
        &\partial_1y\cdot\partial_1y = D^\e_1\cdot D^\e_1,\\
        &\partial_1y\cdot\frac{\partial_2y}{\e} = D^\e_1\cdot D^\e_2,\\
        &\frac{\partial_2y}{\e}\cdot\frac{\partial_2y}{\e} = D^\e_2\cdot D^\e_2,
\end{aligned}
    \end{cases}
    \end{equation}
and
\begin{equation}
\label{rescaled_boundary_cond_2}
    \begin{cases}
    \begin{aligned}
        &y(0, 0) = (0, 0, 0) \text{ and }y(\ell, 0) = \bar{y},\\
        &\nabla y(0, x_2) = (e_1|e_2)\nabla\chi_\e(0, x_2)\text{ for }x_2\in (-\e/2, \e/2),\\
        &\nabla y(\ell, x_2) = (\bar{d}_1|\bar{d}_2)\cdot \nabla\chi_\e(\ell, x_2) \text{ for }x_2\in (-\e/2, \e/2).
\end{aligned}
    \end{cases}
\end{equation}
In particular, $u\in \A_\e$ if and only if the rescaled deformation $y$ belongs to the space 
\begin{equation}
    \A_\e^\Omega\coloneqq\{y\in W^{2, 2}(\Omega:\R^3):y\text{ satisfies }(\ref{rescaled_boundary_cond_1})\text{ and }(\ref{rescaled_boundary_cond_2})\}.
\end{equation}
The normal vector to the rescaled configuration is 
\(\nu_y \coloneqq \nu_u\circ\chi_\e,\)
and following the same computations performed in \cite[Section 2.1]{mora2}, we find that
\(\displaystyle\Pi_u\circ\chi_\e = (D^\e)^{-T}\Pi_{y, \e}(D^\e)^{-1},\)
where $\Pi_{y, \e}$ is
the rescaled second fundamental form of $y$ defined as
\begin{equation}\label{eq:scaled_second_form}
\Pi_{y, \e} = (\nu_y\cdot\partial_1\partial_1y)e_1\otimes e_1+\nu_y\cdot\frac{\partial_1\partial_2y}{\e}(e_1\otimes e_2+e_2\otimes e_1)+\nu_y\cdot\frac{\partial_2\partial_2y}{\e^2}e_2\otimes e_2.
\end{equation}
Finally, the energy of the strip can be written in terms of the rescaled deformation as $J_\e:\A_\e^\Omega\to[0, +\infty)$, defined by
\begin{equation}\label{eq:rescaled_energy}
    J_\e(y) = \int_\Omega\K\bigl((D^\e)^{-T}(\Pi_{y, \e}-\Pi^0_\e)(D^\e)^{-1}\bigl)\cdot \bigl((D^\e)^{-T}(\Pi_{y, \e}-\Pi^0_\e)(D^\e)^{-1}\bigl)\det D^\e\dx,
\end{equation}
where we have set
\[\Pi^0_\e\coloneqq ((D^\e)^T)\Pi^{nat}_\e\circ\chi_\e D^\e.\]
In this way, we have $J_\e(y) = E_\e(u).$

\subsection{Compactness and limit space}\label{subsec:compactness_and_limit_space}
In the following we will always assume that
\[\Pi^0_\e\to \Pi^0\text{ strongly in }L^2(\Omega;\R^{2\times 2}),\]
where $\Pi^0 \in L^2(I;\rsym).$ 
Moreover we set $D_\alpha\coloneqq De_\alpha$ for $\alpha = 1, 2$.
\paragraph{Limit space}
As $\e$ tends to zero, the convergence of the admissible deformations naturally leads, as shown in Lemma \ref{lemma:compactness} below, to the class of limiting admissible configurations
\begin{equation}
\label{eq:limit_space}
\begin{aligned}
\A_0 &= \Bigl\{(y, R)\in W^{2, 2}(I;\R^3)\times W^{1, 2}(I;\R^{3\times3}):R^T = (d_1|d_2|d_3),\\
&d_1 = y',\;\;d_3 = \frac{d_1\wedge d_2}{|d_1\wedge d_2|},\;\;d_\alpha\cdot d_\beta = D_\alpha\cdot D_\beta,\;\;d_1'\cdot(d_3\wedge d_1) = D_1'\cdot(e_3\wedge D_1),\\
& y(0) = 0,\;\;y(\ell) =\bar{y},\;\;d_i(0) = D_i(0)\text{ for }i=1, 2,\;\;\Bigl(d_1(\ell)|d_2(\ell)\Bigl) = (\bar{d}_1|\bar{d}_2)\cdot\nabla\chi(\ell, 0)\Bigl\},
\end{aligned}    
\end{equation}
where in the equalities $d_i(0) = D_i(0)$ we identify $D_i(0)$ with a vector in $\R^2\times\{0\}.$

\paragraph{Compactness}

We define the scaled gradient $\nabla_\e$ by $\nabla_\e\cdot = (\partial_1\cdot|\e^{-1}\partial_2\cdot).$
\begin{lemma}\label{lemma:compactness}
    Let $(y_\e)_\e$ be a sequence of scaled isometries such that $y_\e\in\A^\Omega_\e$ for every $\e>0$ and 
\[\sup_\e J_\e(y_\e)<\infty.\]
Then, up to a subsequence, there exists $(y, R)\in\A_0$ such that, setting $R = (d_1|d_2|d_3)^T,$
\begin{equation}\label{eq:lemma_compact_1}
    y_\e\rightharpoonup y\text{ in }W^{2, 2}(\Omega;\R^3),\;\nabla_\e y_\e\rightharpoonup(d_1|d_2)\text{ in }W^{1, 2}(\Omega;\R^{3\times 2}),
\end{equation}
and 
\begin{equation}\label{eq:lemma_compact_2}
    A_{y_\e, \e}\rightharpoonup\begin{pmatrix}
    d_1'\cdot d_3&d_2'\cdot d_3\\
    d_2'\cdot d_3&\gamma
\end{pmatrix}\text{ in }L^2(\Omega;\rsym)
\end{equation}
for some $\gamma\in L^2(\Omega).$
\end{lemma}
\begin{proof}
    Following the same proof as in \cite[Lemma 2]{mora2}, we deduce that exists $(y, R)\in W^{2, 2}(I;\R^3)\times W^{1, 2}(I;\R^{3\times 3})$ and a subsequence $(y_{\e_n})_n$ verifying (\ref{eq:lemma_compact_1}), (\ref{eq:lemma_compact_2}), and all the properties in the second line of (\ref{eq:limit_space}). To conclude that $(y, R)\in\A_0$ we only need to verify that the boundary conditions at $0$ and $\ell$ are satisfied.

    Since $W^{2, 2}(\Omega;\R^3)$ is compactly embedded in $C(\bar{\Omega}),$ the conditions $y(0) = 0$ and $y(\ell) = \bar{y}$ follow passing to the limit in the relations $y_\e(0, 0) = 0$ and $y_\e(\ell, 0) = \bar{y}$.

    Moreover, we have $\partial_1y_\e(0, x_2) = (e_1|e_2)\cdot\partial_1\chi(0, \e x_2)$. Since the right hand side converges uniformly to $(e_1|e_2)\cdot\partial\chi(0, 0)$ and $\partial_1y_\e\rightharpoonup d_1$ in $W^{1, 2}(\Omega;\R^3),$ by the continuity of the trace we deduce $d_1(0) = D_1(0).$ Similarly, from $\partial_2y_\e(0, x_2)/\e = (e_1|e_2)\cdot\partial_2\chi(0, \e x_2)$ we deduce that $d_2(0) = D_2(0)$. In a similar way we obtain the boundary conditions for $d_1$ and $d_2$ at $\ell$.
\end{proof}

\paragraph{Constraints on the boundary data.} For $\chi(x_1, x_2) = (x_1, x_2)$ and $|\bar{y}|<\ell$, one can show that $\A_0\neq \emptyset,$ see \cite{mora1}. For a general $\chi$, a complete characterization of the boundary data $\bar{y}, \bar{R}$ for which $\A_0\neq\emptyset,$ is not available at the moment. In the following, we will assume a priori the class $\A_0$ is not empty. It is easy to construct boundary data $\bar{y}, \bar{R}$ for which this condition is satisfied. Indeed, if $R = (d_1|d_2|d_3)^T$ is as in (\ref{eq:limit_space}), then $\hat{R}\coloneqq(d_1|d_3\wedge d_1|d_3)^T$ solves
\begin{equation}\label{eq:6}
    \hat{R}(0) = \eye,\;\;\;\hat{R}' = \begin{pmatrix}
    0&k&\mu\\
    -k&0&\tau\\
    -\mu&-\tau&0
\end{pmatrix}\hat{R},
\end{equation}
where $k(t)\coloneqq D_1'(t)\cdot(e_3\wedge D_1)$ and $\mu, \tau$ are some functions in $L^2(I)$. Therefore, if we choose $\bar{\mu}, \bar{\tau}\in L^2(I)$, then we can define $\hat{R}$ to be the solution of (\ref{eq:6}) with data $\bar{\mu}, \bar{\tau}$, and then define $R = (d_1|d_2|d_3)^T$ as
follows\[d_1 = \hat{d}_1,\;d_3 = \hat{d}_3,\;d_2 = (D_1\cdot D_2)d_1+\alpha d_3\wedge d_1,\]
where $\alpha = \sqrt{|D_2|^2-(D_1\cdot D_2)^2}$. Defining $y(t) = \int_0^td_1'(s)\ds$ we conclude that the set $\A_0$, with boundary data $\bar{y} = y(\ell)$ and $\bar{R} = R(\ell)$, is nonempty.

\paragraph{Gamma-convergence Theorem.}
We define the constants
\begin{equation}\label{alpha_p}
    \alpha^+\coloneqq\sup\{\alpha>0:\:Q(M)+\alpha\det(M)\ge 0\text{ for every }M\in\rsym\},
\end{equation}
and 
\begin{equation}\label{alpha_m}
    \alpha^-\coloneqq\sup\{\alpha>0:\:Q(M)-\alpha\det(M)\ge 0\text{ for every }M\in\rsym\}.
\end{equation}
The limit energy density is the function $\bar{Q}:I\times\R\times\R\to[0, +\infty)$ given by
\begin{equation}\label{eq:limit_energy_density}
\begin{aligned}
    \bar{Q}(x_1, \mu, \tau)\coloneqq&\min\Bigl\{Q\bigl(D(x_1)^{-T}(C-\Pi^o(x_1))D(x_1)^{-1}\bigl)\det D(x_1)+\alpha^+\frac{(\det C)^+}{\det D(x_1)}+\alpha^-\frac{(\det C)^-}{\det D(x_1)}:\\
    &C = \mu e_1\otimes e_1+\tau(e_1\otimes e_2+e_2\otimes e_1)+\gamma e_2\otimes e_2,\,\gamma\in\R\Bigl\}
\end{aligned}
\end{equation}
for every $x_1\in I$ and $\mu, \tau\in\R,$ where $(\det C)^- = \max(-\det C, 0)$ and $(\det C)^+ = \max(\det C, 0)$.
The limit functional is $J:\A_0\to\R$ given by
\begin{equation}\label{def:limit_functional}
    J(y, d_1, d_2, d_3) \coloneqq \int_I\bar{Q}(x_1, d_1'\cdot d_3, d_2'\cdot d_3)\dx_1
\end{equation}
for every $(y, R) = \bigl(y, (d_1|d_2|d_3)^T\bigl)\in\A_0.$

\begin{remark}\label{rem:en_density}
    Setting $D^\alpha\coloneqq D^{-T}e_\alpha$ we obtain $D^\alpha\cdot D_\beta = \delta_{\alpha\beta}.$ Therefore, $\bar{Q}$ can also be represented as
    \begin{equation}\label{eq:new_limit_energy_density}
        \begin{aligned}
    \bar{Q}(x_1, \mu, \tau)\coloneqq&\min\Bigl\{\bigl(Q(M-D^{-T}\Pi^oD^{-1})+\alpha^+(\det M)^++\alpha^-(\det M)^-\bigl)\det D:\\
    &M = \mu D^1\otimes D^1+\tau(D^1\otimes D^2+D^2\otimes D^1)+\gamma D^2\otimes D^2,\,\gamma\in\R\Bigl\}.
\end{aligned}
    \end{equation}
    
\end{remark}

We are now ready to state the $\Gamma-$convergence result.
\begin{theorem}\label{teo:gamma_convergence}
Assume that the space $\A_0$ is not empty and that either $\bar{y}\neq\chi(\ell, 0)$ or $\bar{R}^T\neq(D_1(\ell)|D_2(\ell)|e_3)$. Then, as $\e\to 0$, the functionals $J_\e$ $\Gamma-$converge to the limit functional $J$ in the following sense:
\begin{itemize}
    \item[(i)]\emph{(liminf inequality)} for every sequence $(y_\e)_\e$ such that $y_\e\in\A_\e$ for all $\e$, and for every $\bigl(y, (d_1|d_2|d_3)^T\bigl)\in\A_0$ such that $y_\e\rightharpoonup y$ weakly in $W^{2, 2}(\Omega;\R^3)$ and $\nabla_\e y_\e\rightharpoonup (d_1|d_2)$ weakly in $W^{1, 2}(\Omega;\R^3)$, we have that
    \[\liminf_{\e\to 0}J_\e(y_\e)\ge J(y, d_1, d_2, d_3);\]
    \item[(ii)]\emph{(recovery sequence)} for every $\bigl(y, (d_1|d_2|d_3)^T\bigl)\in\A_0$ there exists a sequence $(y_\e)_\e$ with $y_\e\in\A_\e$ for every $\e$, such that 
    \[y_\e\rightharpoonup y \text{ weakly in }W^{2, 2}(\Omega;\R^3) \text{ and }\nabla_\e y_\e\rightharpoonup (d_1|d_2) \text{ weakly in }W^{1, 2}(\Omega;\R^3),\]
    and 
    \[\limsup_{\e\to 0}J_\e(y_\e)\le J(y, d_1, d_2, d_3).\]
\end{itemize}
\end{theorem}
The requirement that either $\bar{y}\neq\chi(\ell, 0)$ or $\bar{R}^T\neq(D_1(\ell)|D_2(\ell)|e_3)$ can be slightly relaxed. We provide a discussion of this point in Remark \ref{rem:degenerate_case}, following the proof of Theorem \ref{teo:gamma_convergence} in Section \ref{sec:recovery}.

\begin{remark}\label{rem:non_empty_space}
    A consequence of Theorem \ref{teo:gamma_convergence}-(ii) is that, if $\A_0\neq\emptyset$, then $\A^\Omega_\e\neq\emptyset$ for every $\e$ sufficiently small.
\end{remark}

\section{Refinement of the relaxation result}\label{sec:refinement}
\subsection{Notation}\label{subsec:notation}
We set $\ell = 1$ and $I = (0, 1)$: this will not affect any result of this section. In the following we will identify vectors $(\alpha, \beta)\in\R^2$ with $(\alpha, \beta, 0)\in\R^3$ and vice versa. In this sense, recalling $B(\cdot)\coloneqq\chi(\cdot, 0)$, we define $N\coloneqq e_3\wedge B'$.

For maps $v\in L^2(I;\R^2)$ we denote by $v^\B_i$ the components of $v$ the basis $\B\coloneqq\{B', N\}$ of $\R^2$, that is, we write 
\[v(t) = v^\B_1(t)B'(t)+v^\B_2(t)N(t),\text{ and we set }v^\B\coloneqq (v^\B_1, v^\B_2)^T.\]
Similarly, for $M\in L^2(I;\rsym)$ we call $M^\B_{ij}$ the entries of $M$ with respect to the basis $\B$, that is,
\begin{equation}\label{eq:notation_M}
        M = M^\B_{11}B'\otimes B' +M_{12}^\B\bigl(B'\otimes N+N\otimes B'\bigl)+M_{22}^\B N\otimes N, \text{ and we set }M^\mathcal{B}\coloneqq\begin{pmatrix}
    M_{11}^\B&M^\B_{12}\\M_{12}^\B&M^\B_{22}
\end{pmatrix}.
    \end{equation}
Sometimes for $M\in\rsym$ we will also write $M^{\B(t)}$ to denote the entries of $M$ with respect to the basis $\B$ at time $t$. 
We set 
\[\U\coloneqq\spn\{e_1\otimes e_3-e_3\otimes e_1,e_2\otimes e_3-e_3\otimes e_2\},\;\;\tilde{\U}:=\mathbb{M}_{skew}^{3\times 3}.\]
For a matrix $A\in L^2(I;\tilde{\U})$ we define $R_A:I\to SO(3)$ to be the solution of the Cauchy problem
\begin{equation}\label{eq:cauchy_problem}
    \begin{cases}
    R_A'(t) = A(t)R_A(t)\text{ for }t\in I,\\
    R_A(0) = \eye,
\end{cases}
\end{equation}
so that $R_A\in W^{1, 2}(I;SO(3)).$
For a given matrix field $M\in L^2(I,\rsym)$, we denote $A_M\in L^2(I, \tilde{\U})$ the matrix
\[A_M = \begin{pmatrix}
    0&k&M^\B_{11}\\
    -k&0&M^\B_{12}\\
    -M^\B_{11}&-M^\B_{12}&0
\end{pmatrix},\]
where $k\in L^2(I;\R)$ is defined as $k(t)\coloneqq B''(t)\cdot N = D_1'\cdot(e_3\wedge D_1)$.

\begin{definition}
    \label{def:non_deg}
    We say that a map $A\in L^2(I;\tilde{\U})$ is \emph{nondegenerate} on a set $J\subseteq I$ if either $J\cap \{A_{13}\neq 0\}$ or $J\cap \{A_{23}\neq 0\}$ have positive measure. Whenever we refer to a map as \emph{nondegenerate} without specifying the interval 
    $J$, we implicitly assume that it is nondegenerate on $I$.
\end{definition}

\begin{definition}
    \label{def:admissible}
    Let $\bar{y}\in\R^3$ and $\bar{R}\in SO(3)$. For $A\in L^2(I;\tilde{\U})$ we define 
    \[\Gamma_A\coloneqq\int_0^1 R_A^T(t)e_1\dt.\]
    We say that $A$ is \emph{admissible} for the data $\bar{y}$ and $\bar{R}$ if it satisfies \[A_{12}\equiv k \;\text{ a.e. on I,}\;\;R_A^T(1) =\bar{R}^T\cdot(B'(1)|N(1)|e_3\bigl)\text{ and }\Gamma_A = \bar{y}.\]
\end{definition}

\begin{remark}
    Let $(y, (d_1|d_2|d_3)^T)\in\A_0$, set
    \(\mu = d_1'\cdot d_3,\;\tau = d_2'\cdot d_3,\) and define
    \[M\coloneqq\mu D^1\otimes D^1+\tau(D^1\otimes D^2+D^2\otimes D^1)+\gamma D^2\otimes D^2\]
    where $\gamma\in L^2(I)$.
    In the proof of \cite[Theorem 5-(ii)]{mora2} the authors showed that the directors 
    \[\tilde{R}^T = (\tilde{d}_1|\tilde{d}_2|\tilde{d}_3)\coloneqq (d_1|d_3\wedge d_1|d_3)\]
    solve the ODE system
    \(\displaystyle\tilde{R}' = A_M\tilde{R}\)
    with initial condition $\tilde{R} = \eye$ and that $A_M$ is admissible for the data $\bar{y}$ and $\bar{R}$ according to Definition~\ref{def:admissible}.
\end{remark}

Finally, we introduce the functional 
\begin{equation}\label{eq:functional_F}
    \F:L^2(I;\rsym)\to[0, +\infty),\;\;\F(M) = \int_I\bigl(Q(M)+\alpha^+(\det M)^++\alpha^-(\det M)^-\bigl)\det(D(x))\dx.
\end{equation}

Let $Q^0:\rsym\to[0, +\infty)$ be the function 
\begin{equation*}
    Q^0(M) = \begin{cases}
        Q(M)&\text{ if }\det(M) = 0,\\
        +\infty & \text{otherwise},
    \end{cases}
\end{equation*}
and let $\hat{\F}$ be defined as
\[\hat{\F}:L^2(I;\rsym)\to[0, +\infty],\;\;\hat{\F}(M) = \int_IQ^0(M(x))\dx.\]
By \cite[Proposition 9]{mora2} the bipolar function of $Q^0$ is given by
\begin{equation}\label{eq:qstar}
    Q^{**}(M)\coloneqq Q(M)+\alpha^+(\det M)^++\alpha^-(\det M)^-,
\end{equation} 
where $\alpha^\pm$ are defined as in (\ref{alpha_p}) and (\ref{alpha_m}). In case $\det(D(x)) = 1$ for every $x\in I$ it holds that
\begin{equation}\label{rem:relaxation}
	\F \text{ is the lower semicontinuous envelope of } \hat{\F} \text{ with respect to the weak topology of }L^2(I;\rsym).
\end{equation}

\subsection{The relaxation result}
The main Theorem of this section is the following.
\begin{theorem}\label{teo:perturbation}
    Let $M\in L^2(I;\rsym)$ be such that $A_M$ is nondegenerate and admissible for some data $\bar{y},\,\bar{R}$. Then, there exist $\lambda_n\in C^1(\bar{I})$ and $p_n\in C^1(\bar{I};\mathbb{S}^1)$ such that, setting $M^n \coloneqq \lambda_np_n\otimes p_n,$ we have
    \begin{itemize}
        \item[(i)] $p_n\cdot B'>0$ everywhere on $\bar{I},\;$ $p_n = B'$ near $\partial I$ and $\lambda_n = 0$ near $\partial I$ for every $n$;\label{point:1}
        \item[(ii)] $A_{M^n}$ is admissible for every $n$;\label{point:2}
        \item [(iii)]$M^n\rightharpoonup M$ weakly in $L^2(I;\rsym)$;\label{point:3}
        \item [(iv)]$\F(M^n)\to \F(M).$\label{point:4}
    \end{itemize}
\end{theorem}
The key tool we will use to create admissible $A_{M^n}$ is the following Lemma.
\begin{lemma}\label{lemma:perturbation_hornung}
    Let $A\in L^2(I;\tilde{\U})$ and assume that there is a measurable set $J\subseteq I$ of positive measure such that $A$ is not degenerate on $J$. Then every $L^2-$dense subspace of
    \[\tilde{E}\coloneqq \{\hat{A}\in L^2(I;\U):\hat{A} = 0\text{ a.e. in } I\setminus J\}\]
    contains a finite dimensional subspace $E$ such that, whenever $A_n\in L^2(I;\tilde{\U})$ converges weakly in $L^2(I;\tilde{\U})$ to $A$, then there exists a sequence $(\hat{A}_n)_n\subseteq E$ converging to zero in $E$ and such that
    \[R_{A_n+\hat{A}_n}(1) = R_A(1)\;\;\text{ and }\;\;\Gamma_{A_n+\hat{A}_n} = \Gamma_A\]
    for every $n$ large enough.
\end{lemma}
\begin{proof}
    The claim follows from \cite[Theorem 3.2]{hornung}, since if $A$ is not degenerate according to our Definition \ref{def:non_deg}, then $A$ is not degenerate also in the sense of \cite[Definition 3.1]{hornung}, because of \cite[Proposition 3.3-(ii)]{hornung}. We only point out that, since $\hat{A}^n$ takes values in $\U$, if $(A^n)_{12} = k$, then $(A^n + \hat{A}^n)_{12} = k$ as well.
\end{proof}
\begin{remark}\label{rem:uniform_convergence}
If $\tilde{E}\subseteq L^2(I;\U)\cap L^\infty(I;\U)$ then, since $E$ is finite dimensional, the sequence $\hat{A}_n$ provided by Lemma \ref{lemma:perturbation_hornung} converges to zero uniformly.
\end{remark}
Before proving Theorem \ref{teo:perturbation}, which is the original contribution of this work, we explain the strategy of the proof.

Suppose for simplicity that $\det(D) = 1$ in $I$, and let $M$ be as in Theorem \ref{teo:perturbation}. From \eqref{rem:relaxation} the functional $\F$ is the lower semicontinuous envelope of $\hat{\F},$ therefore there exists $(M^n)_n$ that verify statements (iii) and (iv) of Theorem \ref{teo:perturbation}, and such that $\det(M^n) = 0$ in $I$ for every $n$. In view of this last property there exist $\lambda_n,\,p_n$ such that $M^n = \lambda_np_n\otimes p_n.$ Statement (i) of Theorem \ref{teo:perturbation} can easily be achieved by regularizing $\lambda_n$ and $p_n$. The main contribution of this work are Lemma \ref{lemma:approx2} and Theorem \ref{teo:approx1}, where we give an explicit construction of $(M^n)_n$ that satisfies simultaneously properties (ii)-(iv) when the limit function $M$ is piecewise constant.

First we explain how to obtain properties (iii)-(iv). Suppose for simplicity that $M$ is constant on $I$ and $\det(M)>0$. The quadratic form $\qstar$ defined in \eqref{eq:qstar} is positive semidefinite on the region $\{C\in\rsym:\det(C)>0\}.$ For $t\in I$ we denote by $\qstart$ the associated quadratic form written with respect to the basis $\B$. Let $v(t)$ be an eigenvector associated with the zero eigenvalue. In Lemma \ref{lemma:approx2} we prove that there exists $s_1<0<s_2$ such that, setting $M_i(t):=M+s_iv(t)$ for $i=1, 2$, one has $\det(M_1(t)) = 0$ and $\qstart(M_1(t)) = \qstart(M) = \qstart(M_2(t))$. In the proof of Theorem \ref{teo:approx1} we show that each element $M^n$ of the sequence satisfying (iii)-(iv) can be constructed  by oscillating properly between the values $M_1(t)$ and $M_2(t)$.
    
It remains to show that the sequence $(M^n)_n$ verifies (ii). By Lemma \ref{lemma:perturbation_hornung} we can modify $A_{M^n}$ to obtain admissible functions $\tilde{A}^n$. To complete the proof, we modify $M^n$ in order to build a sequence $\tilde{M}^n$ whose elements have zero determinant, verify statements (iii), (iv), and such that $\tilde{A}^n = A_{\tilde{M}^n}$. We observe that if $\tilde{A}^n = A_{\tilde{M}^n}$ then, denoting by $\tilde{M}^{\B, n}_{ij}$ the components of $\tilde{M}^n$ with respect to the basis $\B$, one has
$\tilde{M}^{\B, n}_{11} = \tilde{A}^n_{13}$ and $M^{\B, n}_{12} = \tilde{A}^n_{23}$. Enforcing $\det(\tilde{M}^n) = 0,$ we deduce that $\tilde{M}^{\B, n}_{22} = \frac{(\tilde{A}^n_{23})^2}{\tilde{A}^n_{13}}$, whenever $\tilde{A}^n_{13}\neq 0.$ In order to guarantee that $\tilde{M}^n\weak M$ weakly in $L^2(I;\rsym)$ and $\F(\tilde{M}^n)\to\F(M),$ it is crucial to ensure that $\tilde{A}^n_{13}$ is uniformly bounded away from zero. Since $\tilde{A}^n_{13}$ is a perturbation of $A^n_{13} = M^{\B,n}_{11},$ this condition amounts to require that $M^{\B,n}_{11}$ is uniformly bounded away from zero. 
    
For $t\in I$ we identify $v(t)$ with the vector $(v_{11}(t), v_{22}(t), 2v_{12}(t))$, and we denote by $\pi(t):=\text{span}\{e_2, v(t)\} = \{(x, y, z)\in\R^3:\alpha(t)x+\beta(t)z = 0\}$ for some functions $\alpha(t)$ and $\beta(t)$. In the proof of Theorem \ref{teo:approx1} we show that, if 
    \begin{equation}\label{eq:approx_assumptions}
        |M^{\B}_{11}|\ge c\; \text{ a.e. in }I,\;\text{ and if }\;|\alpha M^{\B, n}_{11}+2\beta M^{\B, n}_{12}|\ge c\;\text{ a.e. in }I,
    \end{equation} 
    then it is possible to construct $M^n$ in such a way that also $M^{\B, n}_{11}$ is uniformly bounded away from zero. In Lemma \ref{lemma:approx5} we show that by approximation we can always reduce to the case where \eqref{eq:approx_assumptions} holds.

\subsection{Avoiding a moving plane through perturbation}
\label{subsec:avoiding_plane}
The goal of this section is to prove Lemma \ref{lemma:approx5}, which is a technical approximation result. This will be useful in the proof of Theorem \ref{teo:perturbation} to reduce the argument to a map $M$ satisfying some additional properties (see the discussion after Remark \ref{rem:uniform_convergence} for more details).

For $(\alpha, \beta)\in\R^2\setminus\{(0, 0)\}$ we call
    \[\pi_{\alpha, \beta}\coloneqq\{(x, y, z)\in\R^3: \alpha x+\beta z = 0 \}\]
which is a plane in $\R^3$ containing the $y$ axis. For $c>0$ we set
\begin{equation}\label{eq.2}
    \tilde{\U}^{(\alpha, \beta)}_c\coloneqq
\{ A\in\tilde{\mathcal{U}}:|\alpha A_{13}+2\beta A_{23}|\ge c\}\;\text{ and }\;\tilde{\U}_c:=\{A\in\tilde{\mathcal{U}}:|A_{13}|\ge c \text{ and }|A_{23}|\ge c\}.\}
\end{equation}
In this subsection we assume that four functions $\alpha^\pm, \beta^\pm\in C(\bar{I})$ and a constant $r>0$ are given, such that 
\begin{equation}
    ||\bigl(\alpha^+(t),\beta^+(t)\bigl)||>r \;\text{ and }\;||\bigl(\alpha^-(t),\beta^-(t)\bigl)||> r \;\text{ for any }t\in I.
\end{equation} These functions identify two moving planes $t\mapsto\pi_{\alpha^+(t), \beta^+(t)}$ and $t\mapsto\pi_{\alpha^-(t), \beta^-(t)}$. 

Finally, we denote by $\mathcal{P}(I;X)$ the space of piecewise constant functions on the interval $I$ taking values in some space $X$ (in case $X=\R$ we simply write $\mathcal{P}(I))$. Moreover, we define $\mathcal{P}^\B(I;\R^2)\coloneqq\{u\in L^\infty(I;\R^2):u^\B\in\mathcal{P}(I;\R^2)\}$, and $\mathcal{P}^\B(I;\rsym)\coloneqq\{M\in L^\infty(I;\rsym):M^\B\in\mathcal{P}(I;\rsym)\}.$ With an abuse of notation, for admissible maps $A\in L^\infty(I;\tilde{\U})$, we write $A\in \mathcal{P}(I;\tilde{\U})$ if $A_{13}$ and $A_{23}$ belong to $\mathcal{P}(I)$, despite the fact that $A_{12} = k\not\in\mathcal{P}(I).$

Before stating the main lemma of this section, we state and prove the following auxiliary result.
\begin{lemma}\label{lemma:3.6_Mora}
    \begstat{} Let $A\in L^2(I;\tilde{\U})$ be nondegenerate and admissible for data $\y$ and $\RR$. Then, there exist $c_n>0$ and admissible $A^n\in\mathcal{P}(I;\tilde{\U}_{c_n})$ such that $A^n\to A$ strongly in $L^2(I;\tilde{\U}).$
\end{lemma}
\begin{proof}
    The proof is a straightforward adaptation of \cite[Lemma 3.6]{mora1}. We only sketch the main steps.
    \begin{itemize}
        \item[(i)] There exist admissible $A^n\in L^2(I;\tilde{\U})$ such that $A^n\to A$ strongly in $L^2(I;\tilde{\U})$ and $A^n_{13}\cdot A_{23}^n\neq 0$ on a set of positive measure independent of $n$.
        \item[(ii)] There exist admissible $A^n\in L^2(I;\tilde{\U}_{1/n})$ such that $A^n\to A$ strongly in $L^2(I;\tilde{\U}).$
    \end{itemize}
    The first result can be proved as \cite[Lemma 3.4]{mora1} invoking Lemma \ref{lemma:perturbation_hornung} instead of \cite[Lemma 3.2]{mora1}. Analogously, (ii) can be proved as \cite[Lemma 3.5]{mora1} invoking (i) instead of \cite[Lemma 3.4]{mora1} and Lemma \ref{lemma:perturbation_hornung} instead of \cite[Lemma 3.2]{mora1}. We only point out that given $A_n\rightharpoonup A$ weakly in $L^2(I;\tilde{\U})$, Lemma \ref{lemma:perturbation_hornung}
    allows one to find a perturbed sequence 
     $A_n+\hat{A}_n$ such that $\hat{A}_n$ belongs to $L^2(I;\U).$ In particular this implies $(A_n+\hat{A}_n)_{12} = (A_n)_{12}$ a.e. on $I$, which is a necessary condition for the perturbed sequence to be admissible.
\end{proof}

We now state the main approximation Lemma. For simplicity, we write $A\in\mathcal{P}\bigl(I;\tilde{\U}_c\cap\tilde{\U}^{(\alpha^+, \beta^+)}_{c_n}\cap\tilde{\U}^{(\alpha^-, \beta^-)}_{c_n}\bigl)$ meaning that $A(x)\in \tilde{\U}_c\cap\tilde{\U}^{(\alpha^+(x), \beta^+(x))}_{c_n}\cap\tilde{\U}^{(\alpha^-(x), \beta^-(x))}_{c_n}$ for all $x\in I$. 
\begin{lemma}\label{lemma:approx5}
    \begstat{} Let $M\in L^2(I;\rsym)$ be such that $A_M$ is nondegenerate ad admissible for the data $\y$ and $\RR$. Then there exist constants $c_n>0$ and $M^n\in\mathcal{P}^\B(I;\rsym)$ such that $A_{M^n}$ is admissible for all $n$, $A_{M^n}\in \mathcal{P}(I;\tilde{\U}_{c_n}\cap\tilde{\U}^{(\alpha^+, \beta^+)}_{c_n}\cap\tilde{\U}^{(\alpha^-, \beta^-)}_{c_n})$ for every $x\in I$, and $M^n\to M$ strongly in $L^2(I;\rsym).$ In particular, $\F(M^n)\to\F(M)$ as $n\to\infty.$
\end{lemma}
\begin{proof}
    By Lemma \ref{lemma:3.6_Mora} there exists $(c_n)_n$ and admissible $(A^n)_n$ with $A^n\in\mathcal{P}(I;\tilde{\U}_{c_n})$ such that $A^n\to A_M$ strongly in $L^2(I;\tilde{\U})$. We consider an approximating sequence $(M^n)_n$ defined by     \begin{equation}\label{eq:4.30}
        M^{\B,n}:=\begin{pmatrix}
        A^n_{13}&A^n_{23}\\
        A^n_{23}&M^{\B, n}_{22}
    \end{pmatrix},
    \end{equation}
    where $M^{\B, n}_{22}\in\mathcal{P}(I)$ is chosen such that $M^{\B, n}_{22}\to M^\B_{2, 2}$ strongly in $L^2(I).$ Since $A^n\to A$, $M^n\to M$ strongly in $L^2(I;\rsym),$ and, as a consequence, $\F(M^n)\to\F(M).$By a diagonal argument we can therefore suppose that $M^\B_{22}\in\mathcal{P}(I)$ and $A_M\in\mathcal{P}(I;\tilde{\U}_c)$ for some $c>0$. 
    
    Since $A_M$ does not depend on $M^\B_{22},\,$ we can choose the approximating sequence $M^n$ so that $M^{\B,n}_{22} = M^\B_{22}$ for every $n$. In particular, if we find a sequence $(A^n)_n$ with all the desired properties strongly converging to $A_M$, then the lemma is proved by defining $M^n$ as in \eqref{eq:4.30}, with $M^\B_{22}$ in place of $M^{\B,n}_{22}$.

    \textbf{First step.} We rename $A:=A_M$. We claim that there exists a subinterval $J\subseteq I$, a constant $\tilde{c}>0$ and admissible $\tilde{A}^n\in \mathcal{P}(I;\tilde{\U}_{\tilde{c}})$, such that $\tilde{A}^n\to A$ strongly in $L^2(I;\tilde{\U})$ and such that for $n$ sufficiently large it holds $\tilde{A}^n(x)\in \tilde{\U}^{(\alpha^+(x), \beta^+(x))}_{1/n}\cap\tilde{\U}^{(\alpha^-(x), \beta^-(x))}_{1/n}$ for a.e. $x\in J$.

    For every subset $J\subseteq I$ such that $A_{13}$ and $A_{23}$ are constant on $J$, we define 
    \begin{equation}\label{eq:inf}
    \mathcal{I}^+_J\coloneqq\inf\{|A_{13}\alpha^+(x)+2A_{23}\beta^+(x)|:x\in J\}, \text{ and }\mathcal{I}^-_J\coloneqq\inf\{|A_{13}\alpha^-(x)+2A_{23}\beta^-(x)|:x\in J\}.
    \end{equation}
    By \eqref{eq:inf} and the uniform continuity of $\alpha^\pm$ and $\beta^\pm$ there exist $\delta>0$ such that, if $J\subseteq I$ is an interval with length smaller than $\delta,$ then either $|\alpha^+(x)|\ge r/2$ for every $x\in J$ or $|\beta^+(x)|\ge r/2$ for every $x\in J$, and also $|\alpha^-(x)|\ge r/2$ for every $x\in J$ or $|\beta^-(x)|\ge r/2$ for every $x\in J$.
    
    Suppose first that there exists an interval $J$ such that $|J|<\delta$ and (without loss of generality) $\mathcal{I}^+_J>0.$ Now, if there exists $x\in J$ such that $A_{13}\alpha^-(x)+2A_{23}\beta^-(x)\neq 0$ on $J$ then, up to taking a smaller interval, we can assume that also $\mathcal{I}^-_J>0$. In this case we set $\tilde{A}^n\coloneqq A$ and the claim is proved.
    Otherwise, we have that $A_{13}\alpha^-(x)+2A_{23}\beta^-(x) = 0$ for every $x\in J$. Without loss of generality, we can assume that $\alpha^-\ge r/2$ on $J$. We define $\tilde{A}^n\in L^2(I;\tilde{\U})$as
    \[\tilde{A}^n_{13} = A_{13}+\frac{2}{rn}\chi_J,\]
    keeping the other components of $A$ unchanged. In this way we have that for $n$ sufficiently large and for $x\in J$
    \begin{equation}\label{eq:3}
    |\tilde{A}^n_{13}\alpha^++2\tilde{A}^n_{23}\beta^+|\ge 1/n,\;|\tilde{A}^n_{13}\alpha^-+2\tilde{A}^n_{23}\beta^-|\ge 1/n,\;\text{ and }|\tilde{A}^n_{13}|>c/2.
    \end{equation}
    We are only left to discuss the case where
    \[A_{13}\alpha^+(x)+2A_{23}\beta^+(x) = 0\text{ and }A_{13}\alpha^-(x)+2A_{23}\beta^-(x)=0\text{ for every }x\in I.\]
    Let $J$ be an interval where $A$ is constant. Up to flipping the sign of the vector $(\alpha^+(x), \beta^+(x))\in\R^2$, there exists a vector $v\in\R^2$, orthogonal to $(A_{13}, 2A_{23})$ and such that $\norma{v} = 1,$ $(\alpha^+(x), \beta^+(x)) = \norma{(\alpha^+(x), \beta^+(x))}v,$ and $(\alpha^-(x), \beta^-(x)) = \norma{(\alpha^-(x), \beta^-(x))}v$ for every $x\in J$.

    We define $\tilde{A}^n\in L^2(I;\tilde{\U})$as
    \[(\tilde{A}^n_{13}, \tilde{A}^n_{23})\coloneqq \bigl(A_{13}, A_{23}\bigl)+\chi_J\frac{1}{nr}v\]
    keeping the other components of $A$ unchanged. 
    Also in this case for $n$ sufficiently large and for $x\in J$ it holds
    \begin{equation}\label{eq:280}
    |\tilde{A}^n_{13}\alpha^++2\tilde{A}^n_{23}\beta^+|\ge 1/n,\;|\tilde{A}^n_{13}\alpha^-+\tilde{A}^n_{23}\beta^-|\ge 1/n,\;\text{ and }|\tilde{A}^n_{13}|>c/2.
    \end{equation}
    Clearly, $(\tilde{A}^n_{13}, \tilde{A}^n_{23})$ converges uniformly to $(A_{13}, A_{23}).$

    To conclude the first step, we modify $\tilde{A}^n$ to make it admissible.
    Let $\tilde{E}$ be the set of maps in $\mathcal{P}(I;\U)$ that vanish a.e. on $J$. By Lemma \ref{lemma:perturbation_hornung} there exist $\hat{A}^n\in\Tilde{E}$ converging to zero uniformly and such that $A^n \coloneqq \tilde{A}^n+\hat{A}^n$ is admissible (since $(\hat{A}^n)_{12}=0$ and $(\tilde{A}^n)_{12} = A_{12} = k$). By construction $A^n(x) = \tilde{A}^n(x)$ for every $x$ in $J$, and by  (\ref{eq:3}) and (\ref{eq:280}) $ \tilde{A}^n_{|_J}\in \mathcal{P}(J;\mathcal{U}^{(\alpha^+, \beta^+)}_{1/n}\cap\mathcal{U}^{(\alpha^-, \beta^-)}_{1/n})$ as required. Finally, by uniform convergence $A^n\in\mathcal{P}(I;\tilde{\U}_{\tilde{c}})$ for some $\tilde{c}>0$.
    
    \textbf{Second step.} In view of the first step, we can assume that $A_M =: A \in \mathcal{P}(I;\tilde{\U}_c)$, and that there is $J_1\subseteq I$ such that for a positive constant $c>0$ it holds $A(x)\in \tilde{\U}^{(\alpha^+(x), \beta^+(x))}_{c}\cap\tilde{\U}^{(\alpha^-(x), \beta^-(x))}_{c}$ for every $x\in J_1$. We prove that there exist a constant $\tilde{c}>0$, constants $(c_n)_n$, and an approximating sequence $A^n\in\mathcal{P}(I;\tilde{\U}_{\tilde{c}}\cap\tilde{\U}^{(\alpha^+,\beta^+)}_{c_n})$ such that $A^n$ is admissible for all $n$.

    We set $J_2 = I\setminus J_1$ and $\delta$ as in the first step. For every $n$ we consider a partition $\mathbb{P}_n$ of $J_2$ into intervals of length smaller than $\delta$ such that $A$ is constant on each of them and for every $J\in\mathbb{P}_n$
    \begin{equation}\label{eq:inf_sup}
    \Bigl|A_{13}\alpha^+(x)+2A_{23}\beta^+(x)-A_{13}\alpha^+(y)-2A_{23}\beta^+(y)\Bigl|\le \frac{1}{n^2}\quad\text{for every }x, y\in J.
    \end{equation}
    
    Now, for $n\in \N$, we define $\tilde{A}^n$ so that it is constant on each $J\in\mathbb{P}_n.$ More precisely, for every $J\in\mathbb{P}_n,$ let $A^J_{13}$ and $A^J_{23}$ be such that $(A_{13}, A_{23})\equiv (A^J_{13}, A^J_{23})$ a.e. on $J$, and define $\mathcal{I}^+_J$ as in (\ref{eq:inf}). On the interval $J$, we define
    \[(\tilde{A}^n_{13}, \tilde{A}^n_{23}) = \begin{cases}
        \Bigl(A_{13}^J+\frac{6}{rn}, A_{23}^J\Bigl)&\text{ if }|\alpha^+|\ge r/2\text{ on }J \text{ and }\mathcal{I}^+_J< 1/n,\\
        (A_{13}^J, A_{23}^J+\frac{3}{rn})&\text{ if }|\alpha^+(x)|<r/2\;\text{ for some }x\in J\text{ and } \mathcal{I}^+_J< 1/n,\\
        \Bigl(A_{13}^J, A_{23}^J\Bigl)&\text{ if }\mathcal{I}^+_J\ge 1/n.\end{cases}
        \]
    The other components of $A$ are left unchanged. On $J_1$ we define $\tilde{A}^n = A$. 
    
    Since $A$ is constant on each $J\in\mathbb{P}_n$ we have that $\tilde{A}^n$ is also constant on each $J$. Moreover, $\tilde{A}^n\to A$ uniformly on $I$, and for every $J\in\mathbb{P}_n$ and for every $x\in J$ it holds $|\tilde{A}^n_{13}\alpha^+(x)+2\tilde{A}^n_{23}\beta^+(x)|\ge 1/n$ for all $n$ sufficiently large. Indeed, if $\mathcal{I}^+_J\ge1/n$ there is nothing to prove. If instead $|\alpha^+|\ge r/2$ on $J$ and $\mathcal{I}^+_J<1/n$ then 
    \begin{equation}\label{eq:4.34}
        |\tilde{A}^n_{13}\alpha^+(x)+2\tilde{A}^n_{23}\beta^+(x)|\ge \frac{3}{n}\frac{2}{r}|\alpha^+(x)|-\sup\{|A_{13}^J\alpha^+(x)+2A_{23}^J\beta^+(x)|:x\in J\}\ge \frac{3}{n}-\frac{1}{n}-\frac{1}{n^2}\ge 1/n,
    \end{equation}
    where in the last inequality we used \eqref{eq:inf_sup} to estimate the $\sup$.
    Finally, if $|\alpha^+(x)|<r/2$ for some $x\in J$, then our choice of $\delta$ ensures that $|\beta^+|\ge r/2$ on $J$. Therefore we can argue as in \eqref{eq:4.34}, replacing $\frac{3}{n}\frac{2}{r}|\alpha^+(x)|$ with $\frac{3}{n}\frac{2}{r}|\beta^+(x)|$.
    
    Now let $\tilde{E}$ be the set of maps in $\mathcal{P}(I;\mathcal{U})$ that vanish on $J_2$. By Lemma \ref{lemma:perturbation_hornung} there exist $\hat{A}^n\in\tilde{E}$ converging to zero uniformly and such that $A^n\coloneqq \tilde{A}^n+\hat{A}^n$ is admissible. By construction we have $A^n(x) = \tilde{A}^n(x)\in \tilde{\U}^{(\alpha^+(x), \beta^+(x))}_{1/n}$ for $x\in J_2$.
    For $x\in J_1$ it holds $\tilde{A}^n = A\in \tilde{\U}^{(\alpha^+, \beta^+)}_{c}\cap\tilde{\U}^{(\alpha^-, \beta^-)}_{c}.$ Since $\hat{A}^n$ converges uniformly to zero we have that
    \[\Bigl|\bigl(\tilde{A}^n_{13}+\hat{A}^n_{13}\bigl)\alpha^++2\bigl(\tilde{A}^n_{23}+\hat{A}^n_{23}\bigl)\beta^+\Bigl|\ge \Bigl|\tilde{A}^n_{13}\alpha^++\tilde{A}^n_{23}\beta^+\Bigl| - \Bigl|\hat{A}^n_{13}\alpha^++\hat{A}^n_{23}\beta^+\Bigl|\ge c/2\text{ on }J_1\]
    for $n$ sufficiently large. The step is proved by defining the approximating sequence $(M^n)_n$ as in (\ref{eq:4.30}).
    
    \textbf{Third step.} In view of the argument above, by approximation we can assume that there exists a constant $c>0$ such that
    $A_M=:A\in\mathcal{P}(I;\tilde{\U}_{c}\cap\tilde{\U}^{(\alpha^+(x), \beta^+(x))}_c)$. To conclude the proof we just need to repeat the second step above with $\alpha^-$ and $\beta^-$ in place of $\alpha^+$ and $\beta^+$, using the fact that the approximating sequence $\tilde{A}^n$ converges to $A$ uniformly.
\end{proof}

\subsection{A technical approximation result}\label{subsec:technical_approx_result}
We prove here a technical lemma which will be useful in the proof of Theorem \ref{teo:approx1}.

\begin{lemma}\label{lemma:approx2}
    Let $u\in L^2(I)$. Assume that $|u(x)|\ge c$ for a.e. $x\in I,$ for some constant $c>0$. Then there is a sequence $(u_n)_n\subset \mathcal{P}(I)$ strongly converging to $u$ in $L^2(I)$ such that $|u_n|\ge c$ on $I_n$ and $u_n = 0$ on $I\setminus I_n$ for $n$ sufficiently large. Moreover, if $u\in L^\infty(I)$ we can choose $(u_n)_n\subset L^\infty(I)$ such that $\norma{u_n}_{L^\infty}\le\norma{u}_{L^\infty}.$
\end{lemma}
\begin{proof}
    Assume first that $u\ge c$ (resp. $u\le -c$) a.e. in $I$. If $u\in C(\bar{I})$, the thesis follows by taking $u_n=0$ on $I\setminus I_n$ and $u_n(x) \coloneqq u(k/n)$ on $\bigl(k/n, (k+1)/n\bigl)$ for $k = 1, \dots, n-2$.

    \textbf{Step two.} In the general case, let
    \[S^+\coloneqq\{x\in I:u(x)\ge 0\},\quad S^- \coloneqq \{x\in I:u(x)<0\}.\]
    Since $S^+$ is a measurable set, for every $n$ there is a (countable) family of disjoint open intervals $\mathcal{F}^n = \{A^n_k:k\in\N\}$ such that 
    \[S^+\subseteq\bigcup_{k\in\N}A^n_k\quad\text{and}\quad \Bigl|\bigcup_{k\in\N}A^n_k\setminus S^+\Bigl|\le 1/n.\]
    We can find an index $k_n$ such that
    \[\Bigl|\bigcup_{k\le k_n}A^n_k\setminus S^+\Bigl|\le 1/n\quad\text{and }\quad\Bigl|S^+\setminus \bigcup_{k\le k_n}A^n_k\Bigl|\le 1/(2n).\]
    Set $S_n^+ = \bigcup_{k\le k_n}A^n_k$ and $S_n^- = I\setminus S_n^+$; we observe that $S_n^-$ is itself a finite union of disjoint intervals.\\
    We define
    \[u_n(x) = \begin{cases}
        u(x)&\text{ if }x\in S^+\cap S^+_n\cap I_n\text{ or }x\in S^-\cap S^-_n\cap I_n,\\
        c&\text{ if }x\in S^+_n\cap S^-\cap I_n,\\
        -c&\text{ if }x\in S^-_n\cap S^+\cap I_n,\\
        0&\text{ otherwise}.
    \end{cases}\]
    We have that $u_n\to u$ strongly in $L^2(I);$ indeed, setting $C_n:=\bigl(S_n^+\cap S^-\cap I_n\bigl)\cup\bigl(S_n^-\cap S^+\cap I_n\bigl)$ we have
    \[|C_n| \le \bigl|S^+_n\cap S^-\bigl| + \bigl|S^-_n\cap S^+\bigl| = \bigl|S^+_n\setminus S^+\bigl| + \bigl|S^+\setminus S^+_n\bigl| \le \frac{1}{n}+\frac{1}{2n},\]
    hence, using also that $|I\setminus I_n| = 2/n$, it follows
    \[\int_I|u(x)-u_n(x)|^2\dx\le\int_{I\setminus I_n}|u(x)|^2\dx+\int_{C_n}|c-|u(x)||^2\dx\to 0,\;\text{ as }n\to+\infty.\]
    By definition we have that $u_n = 0$ on $I\setminus I_n$, $u_n(x)\ge c$ for a.e. $x\in S^+_n\cap I_n$, and $u_n(x)\le -c$ for a.e. $x\in S^-_n\cap I_n.$    
    Therefore, by a diagonal argument, we can assume that $u=0$ on $I\setminus I_n$ for some $n\in\N$, and that there is a finite number of disjoint intervals $(I^j)_{j\le N}$ such that $\displaystyle\Bigl|I_n\setminus\Bigl(\bigcup_{j\le N} I^j \Bigl)\Bigl| = 0$, and for every $j\le N$, either $u(x)\ge c$ or $u(x)\le-c$ a.e. on $I^j$. We can now conclude by applying Step 1 to $u_{|_{I^j}}$ for each $j = 1, \dots, N$. 
    
    It is immediate to see that, by construction, if $u\in L^\infty(I)$ we can choose $(u_n)_n\subset L^\infty(I)$ such that $\norma{u_n}_{L^\infty}\le\norma{u}_{L^\infty}$.
\end{proof}

\subsection{A first refinement result}\label{subsec:first_refinement_result}
In the following we will identify each element $M\in\rsym$ with the vector $m = (M_{11}, M_{22}, 2M_{12})\in\R^3$ and viceversa. The vector associated with $M^\B$ will be denoted by $m^\B$. Moreover, with an abuse of notation, we identify $\K$ with the only $3\times 3$ symmetric matrix such that $Q(M) = \innerproduct{\K\cdot m}{m}$ for all $M\in\R^{2\times 2}_{sym}.$

\paragraph{Determinant function.} Since $\mathcal{B}$ is an orthonormal basis, we have that
\[m_1m_2-\frac{1}{4}m_3^2 = \det(M) = \det(M^\B) = m^\B_1m^\B_2-\frac{1}{4}(m^\B_3)^2.\]
In particular, setting 
\begin{equation*}
    \D \coloneqq \begin{pmatrix}
0 & 1/2 & 0\\
1/2 &0 & 0\\
0&0&-1/4
\end{pmatrix},
\end{equation*}
we have $\det(M) = \innerproduct{\D \cdot m}{m} = \innerproduct{\D \cdot m^\B}{m^\B}$.
Given the vector $m\in\R^3$ associated to a matrix $M$, we will sometimes write, with an abuse of notation, $\det(m)$ to refer to $\det(M)$.
We will call $\DD^+, \DD^-, \DD^0$ the regions of $\R^3$ where the determinant function is, respectively, strictly positive, strictly negative, or equal to zero.

\paragraph{Energy in the new basis.}
Identifying the basis $\B$ with a map $\mathcal{B}\in C^1(\bar{I}, \R^{2\times 2})$ defined as $\mathcal{B}(t)\coloneqq \bigl(B'(t)|N(t)\bigl)$, we observe that, for every $t\in\bar{I}$, the function 
\[\mathcal{L}_t: \rsym\to \rsym,\quad 
\mathcal{L}_t(C) \coloneqq \mathcal{B}^{-T}(t)C\B^{-1}\]
satisfies
$\mathcal{L}_t(M^{\B(t)}) = M$,
and is a linear isomorphism. Therefore, for every $t\in\bar{I}$ the map 
\[Q_t:\rsym\to\R,\quad Q_t(C) \coloneqq Q\bigl(\mathcal{L}_t(C)\bigl)\]
which satisfies $Q_t(M^{\B(t)}) = Q(M),$ is a positive definite quadratic form. We call $\K_t\in\R^{3\times 3}$ the positive definite symmetric matrix such that $Q_t(C) = \K_tC\cdot C$. Similarly, we define $Q_t^{**}:\rsym\to\rsym$ as
\begin{equation}\label{eq:energy_density_new_coordinates}
    Q_t^{**}(C)\coloneqq Q_t(C)+\alpha^+\det\bigl(\mathcal{L}_t(C)\bigl)^++\alpha^-\det\bigl(\mathcal{L}_t(C)\bigl)^- = Q_t(C)+\alpha^+\det\bigl(C\bigl)^++\alpha^-\det\bigl(C\bigl)^-.
\end{equation}
We observe that \[\F(M) = \int_IQ_t^{**}(M^\B(x))\det(D(x))\dx.\]

\paragraph{Eigenvalues and eigenvectors}\label{par:eigenspace_eigenvectors}
By the definition of $\alpha^\pm$ in \eqref{alpha_p} and \eqref{alpha_m}, 0 is an eigenvalue of the quadratic forms $\K_t-\alpha^-\D$ and $\K_t+\alpha^+\D$. We denote the corresponding eigenspaces by $V_t^-,\, V_t^+$ respectively. Note that, if $m\in V^-_t$ (resp. $m\in V^+_t$), then $\det(m)>0$ (resp. $\det(m)<0$). Indeed, considering for instance the case $m\in V^-_t$, it holds
\[\innerproduct{\bigl(\K_t -\alpha^-\D\bigl)\cdot m}{m} = 0\iff \innerproduct{\K_t \cdot m}{m} = \alpha^-\innerproduct{\D\cdot m}{m}\implies \innerproduct{\D\cdot m}{m}>0\]
since $\K_t$ is positive definite.

We observe that $\bigl(\K_t\pm\alpha^\pm\D\bigl)(C) = 0\iff \bigl(\K\pm\alpha^\pm\D\bigl)(\mathcal{L}_t(C)) = 0\iff \mathcal{L}_t(C)$ is an eigenvector of $\K\pm\alpha^\pm\D$ with null eigenvalue. In particular, the dimension of $V^-_t$ and $V^+_t$ is constant with respect to $t$. Therefore denoting by $V^\pm$ the eigenspace of the zero eigenvalue of $\K\pm\alpha^\pm\D$, if $V^{\pm} = \text{span}\{v^\pm_i:i=1, \dots, \dim(V^\pm)\}$, then for every $t$ it holds 
$V^\pm_t = \text{span}\{\mathcal{L}_t^{-1}\bigl(v^\pm_i\bigl):i=1, \dots, \dim(V^\pm)\}$. In particular, the vectors $w^\pm_i(t)\coloneqq\mathcal{L}_t^{-1}(v^\pm_i),\,i = 1, \dots, \dim(V^\pm)$ are a basis of $V^\pm_t$ for every $t$, and they belong to $C^1(\bar{I};\R^3)$.

\begin{remark}\label{rem:eigenvalues}
    Let $\alpha$ be either $\alpha^+$ or $\alpha^-,$ and let $t\in I$. Then $(\K_t\pm\alpha\D)_{22} = (\K_t)_{22} = \innerproduct{\K_t\cdot e_2}{e_2}>0$
    since $\K_t$ is positive definite for every $t$. We conclude that $\text{dim}(V^\pm_t)\le 2$, and that the vector $(0, 1, 0) = e_2$ never belongs to $V^\pm_t$.
\end{remark}

\begin{lemma}\label{lemma:due}
        Let $V_t$ denote either $V_t^+$ or $V_t^-$, and let $\DD$ denote the region $\DD^-$ or $\DD^+$, correspondingly. For every $t\in I$ let $v: I\to\R^3$ be such that the map $t\mapsto v(t)$ is continuous and $v(t)\neq 0$ for every $t\in J$, and let $\pi_{v(t)} = \text{span}\{v(t), e_2\}\subseteq\R^3$. Consider a vector $m\in\R^3$, and let $J\subseteq I$ be an open interval such that $m\in \DD\setminus \pi_{v(t)}$ for every $t\in J$. Then there exist two continuous functions $m^1,\,m^2:\: J\to\R^3$ such that \begin{itemize}
            \item[(1)] $\det(m^i(t)) = 0$ for every $t\in J,$ for $i=1, 2$;
            \item[(2)] the vector $m$ is a convex combination of $m^1(t)$ and $m^2(t)$ for every $t\in J$;
            \item[(3)] $\qstart(m^i(t)) = \qstart(m)$ for every $t\in J$ and $i=1, 2$;
            \item[(4)] for every $t\in J$ the first component of $m^i(t)$ is never zero for $i=1, 2$;
            \item[(5)] there exists a constant $c_\infty,$ depending only on $m$ and $v$ such that $||m^i||_{L^\infty(J)}\le c_\infty$ for $i=1, 2.$
        \end{itemize}
\end{lemma}
\begin{proof}    
    We can suppose that $m\in \DD^+$ the other case being analogous.
    For every $t\in J$ consider the function $\varphi_{v(t)}^m$ defined by 
    \[\varphi_{v(t)}^m:\R\to\R^3,\quad \varphi_{v(t)}^m(s) = m+sv(t),\]
    which is a parametrization of the line passing through $m$ with direction $v(t)$. Denoting by $m_i$ and $v_i(t)$ the components of $m$ and $v(t)$ respectively, a short computation shows that 
    \begin{equation*}
        \det(\varphi_{v(t)}^m(s)) = \det(m) + (m_2v_1(t)+m_1v_2(t)-\frac{1}{2}m_3v_3(t))\cdot s+\det(v(t))\cdot s^2.
    \end{equation*}
    Since $\det(v(t))<0$, it follows that 
    \begin{equation}\label{eq:det_v}
        \det(\varphi_{v(t)}^m(0)) = \det(m)>0,\quad\lim_{s\to+\infty}\det(\varphi_{v(t)}^m(s)) = -\infty,\quad\lim_{s\to-\infty}\det(\varphi_{v(t)}^m(s)) = -\infty.
    \end{equation}
    We define $s_1, s_2\in\R$ as
    \begin{equation}\label{eq:s_i}
        s_1(t)\coloneqq \inf\{s>0:\:\det(\varphi_{v(t)}^m(s)) = 0\},\quad s_2(t)\coloneqq \sup\{s<0:\:\det(\varphi_{v(t)}^m(s)) = 0\}
    \end{equation}
    and $m^1(t)\coloneqq m+s_1(t)v(t),\, m^2(t)\coloneqq m+s_2(t)v(t).$ It follows immediately that $\det(m^i(t)) = 0$ and that $m$ is a convex combination of $m_1(t)$ and $m_2(t)$ for every $t\in J$.
    Moreover, for every $t\in J$ and $s_2\le s\le s_1$ we have
    \begin{align*}
        &\qstart(m+sv(t)) = Q_t(m+sv(t))+\alpha^-\det(m+sv(t))^-+\alpha^+\det(m+sv(t))^+ = \\        &=\innerproduct{\bigl(\K_t+\alpha^+\D\bigl)\cdot (m+sv(t))}{m+sv(t)} = \qstart(m)+2\innerproduct{\bigl(\K_t+\alpha^+\D\bigl)\cdot v(t)}{m}\cdot s+\innerproduct{\bigl(\K_t+\alpha^+\D\bigl)\cdot v(t)}{v(t)}\cdot s^2,
    \end{align*}
    which reduces to $\qstart(m)$ since $\bigl(\K_t+\alpha^+\D\bigl)\cdot v(t)=0$.
    
    We now prove (4). By contradiction we suppose that for some $t\in J$, without loss of generality, $m^1(t) = (0, y_1, z_1)$ (i.e. the first component of $m^1(t)$ is zero). Thus, $0 = \det(m^1(t)) = -\frac{1}{4}z_1^2$
    implies that $m^1(t) = (0, y_1, 0) = y_1e_2.$ In particular, $m = y_1e_2-s_1(t)v(t)\in\text{span}\{e_2, v(t)\},$ which gives a contradiction.\\
    From (\ref{eq:det_v}) and (\ref{eq:s_i}), we deduce that $s_i(t_n)$ is bounded, therefore property (5) follows from the continuity of $v$.
    Finally we prove that $t\mapsto m^1(t)$ is continuous, the continuity of $m^2(t)$ being analogous. Let $t_n\to\bar{t}\in J.$ From (5), up to subsequences, we may suppose $s_1(t_n)\to\bar{s}_1$. Thus $m^1(t_n) = m+s_1(t_n)v(t_n)\to m+\bar{s}_1v(\bar{t})\eqqcolon\bar{m}^1(\bar{t})$ and $\det(\bar{m}^1(\bar{t})) = 0$. If $\bar{s}_1\neq s_1(\bar{t}),$ this is a contradiction since $\bar{s}_1\ge 0$ and, by \eqref{eq:s_i}, the function $s\mapsto\det(\varphi_{v(\bar{t})}^m(s)\bigl)$ for $s\ge 0$ is a polynomial of degree two vanishing only at $s_1(\bar{t}).$
\end{proof}

The following theorem is a refinement of the relaxation result recalled in \eqref{rem:relaxation} for matrix fields $M\in \mathcal{P}^{\B}(I;\rsym)$ (see the discussion after Remark \ref{rem:uniform_convergence} for more details). For sequences of matrices $(M^n)_n\subseteq L^2(I;\rsym)$ we denote their components in the basis $\B$ by $M_{ij}^{\B, n}$. We also introduce the notation $I_n \coloneqq (1/n, 1-1/n)$.

\begin{theorem}\label{teo:approx1}
        Let $v^\pm(t)\in C^1(I;\R^3)$ be such that $v^\pm(t)\in V^\pm_t$ and $v(t)\neq 0$ for every $t\in I$, and define \[\pi_{v^\pm(t)} = \text{span}\{e_2, v^\pm(t)\} =: \{(x, y, z)\in\R^3:\:\alpha^\pm(t) x+\beta^\pm(t) z = 0\}.\]
        Let $M\in \mathcal{P}^\B(I;\rsym)$ be such that $A_M\in\mathcal{P}(I;\tilde{\U}_c\cap\tilde{\U}^{(\alpha^+, \beta^+)}_c\cap\tilde{\U}^{(\alpha^-, \beta^-)}_c)$ for some positive constant $c>0$. 
        
        Then there exist functions 
        $\lambda_n\in L^\infty(I)$ and $p_n\in L^\infty(I;\mathbb{S}^1)$ with $\lambda_n = 0$ on $I\setminus I_n$ such that 
        setting $M^n \coloneqq \lambda_n p_n\otimes p_n,$ we have:
        \begin{itemize}
            \item[(1)] $M^n\weakstar M$ weakly$^*$in $L^\infty(I;\rsym)$ and $\F(M^{n})\to\F(M);$
            \item[(2)] there exists a constant $\tilde{c}>0$, independent of $n$, such that $|M^{\B, n}_{11}(x)|\ge\tilde{c}$ for a.e. $x\in I_n$ and for all $n$.
        \end{itemize}
\end{theorem}
\begin{proof}
    \textbf{Step one.} We first prove that there exists a sequence $(M^n)_n$ satisfying (1), (2), and such that $\det(M^n) = 0$ on $I$. 
    
    We show how to construct the sequence $M^n$ on every interval $J$ on which $M^\B$ is constant. So, let $J = (a, b)$ be an interval with this property, and let $M^\B\equiv\bar{M}^\B\in\rsym$ on $J$. If $\det(\bar{M}^\B)=0,$ we define $M^{\B, n}\coloneqq\bar{M}^\B$ on $J$ for every $n$. Otherwise, we identify $\bar{M}^\B$ with a vector $\bar{m}^\B\in\R^3$. By Lemma \ref{lemma:due}, for every $t\in J$ there are $\lambda(t)\in(0, 1)$ and $m_1^\B(t), m^\B_2(t)\in\R^3$ such that the $m_i^\B$ are continuous functions on $J$, $\qstart(m^\B_i(t)) = \qstart(\bar{m}^\B),\;\det(m^\B_i(t)) = 0$ for $i=1, 2$, and $\bar{m}^\B = \lambda(t) m^\B_1(t)+(1-\lambda(t)) m^\B_2(t).$ In particular, $\lambda(t) = \norma{\bar{m}-m_2^\B(t)}/\norma{m_1^\B(t)-m_2^\B(t)}$ is a continuous function on $J$. Moreover, from Lemma \ref{lemma:due}-(5) there exists a constant $c_\infty$ depending only on $|m^\B|$ and $||v^\pm||_{L^\infty(J)}$ such that $\norma{m_i^\B}_{L^\infty(J)}\le c_\infty$. Since the first component of $m_i^\B$ is never zero, by continuity we deduce that $|(m_i^\B)_{1}|>\tilde{c}$ for some constant $\tilde{c}$.\\
    For $n\ge 2$ and $k = 0, \dots, n$ we set $t^n_k\coloneqq a+k/n(b-a)$ and $J^n_k\coloneqq[t^n_k, t^n_{k+1})$. The functions 
    \[m^{\B, n}(t)\coloneqq \sum_{k=0}^{n-1}\chi_{J^n_k}(t)\Bigl(\lambda(t^n_k)m^\B_1(t)+\bigl(1-\lambda(t^n_k)\bigl)m^\B_2(t)\Bigl)\]
    converge to $\bar{m}^\B$ uniformly on $J$, and therefore \[\int_{J}\qstart (m^{\B,n}(t))\det(D(t))\dt\to \int_{J}\qstart (\bar{m}^{\B})\det(D(t))\dt.\]
    Since $A_M(t)\in\tilde{\U}_c^{(\alpha^+(t), \beta^+(t))}\cap \tilde{\U}_c^{(\alpha^-(t), \beta^-(t))}$ for every $t\in J$, the uniform convergence of $M^{\B,n}$ to $\bar{M}^\B$ ensures that, for $n$ sufficiently large, \[A_{M^{\B, n}}(t)\in\tilde{\U}_{c/2}^{(\alpha^+(t), \beta^+(t))}\cap \tilde{\U}_{c/2}^{(\alpha^-(t), \beta^-(t))} \text{ for a.e. }t\in J.\]
    Since $\norma{m^{\B, n}}_{L^\infty(J)}\le 2c_\infty$,up to a diagonal argument, we can assume that the function $t\mapsto \lambda(t)$ is constant on the interval $J$.
    
    So, now let us suppose that $m^\B(t) = \lambda m_1^\B(t)+(1-\lambda)m^\B_2(t)$ for $t\in J = (a, b)$, where $\lambda\in (0, 1)$ is a constant, the functions $m^\B_i$ are continuous on $J$ with $\norma{m^\B_i}\le c_\infty$ for $i = 1, 2$ and $\qstart(m^\B(t)) = \qstart(m_i^\B(t))$ for every $t\in J$. Define $J_\lambda\coloneqq \bigl(a, a+\lambda(b-a)\bigl)$ and denote by $t\mapsto\chi_\lambda(t)$ the $(b-a)-$periodic extension to $\R$ of the characteristic function of $J_\lambda$. Finally for every $n$ we define
    \[\tilde{m}^{\B, n}(t)\coloneqq \chi_\lambda(nt)m^\B_1(t)+\bigl(1-\chi_\lambda(nt)\bigl)m^\B_2(t).\]
    By the Riemann-Lebesgue Lemma it holds $\tilde{m}^{\B, n}\weakstar \lambda m^\B_1+(1-\lambda)m^\B_2 = m^\B$ weakly$^*$ in $L^\infty(J)$. Moreover, we have $\det(\tilde{m}^{\B, n}(t)) = 0,\,$ $|\tilde{m}^{\B, n}(t)|\le c_\infty$ for a.e. $t\in J,$ and \[\int_J\qstart(\tilde{m}^{\B, n}(t))\det(D(t))\dt  = \int_{J}\qstart(m^\B(t))\det(D(t))\dt,\]
    where we used that $\qstart(m_i^\B(t)) = \qstart(m^\B(t))$.\\
    \textbf{Step two.} We show that there exist $\lambda_n$ and $p_n$ satisfying all the required properties.
    Let $(M^n)_n$ be the sequence constructed in Step 1.\\
    For every $n$ there are functions $u^n\in L^4(I;\R^2)$ and $\theta^n\in L^\infty(I;\{\pm1\})$ such that 
    $M^n = \theta^nu^n\otimes u^n$ a.e. in $I$. Since $|\theta^n| = 1$ a.e. on $I$, by Lemma \ref{lemma:approx2} we can find a sequence $\theta^n_k\in\mathcal{P}(I)$ with 
    $\theta^n_k = 0$ on $I\setminus I_k$ and $|\theta^n_k|\ge 1$ on $I_k$ such that $\theta^n_k\xrightarrow{k\to\infty}\theta^n\;\text{ strongly in }\;L^2(I)$.
    By a diagonal argument we can thus assume that $M^n = \theta^nu^n\otimes u^n$, with $\theta^n\in\mathcal{P}(I),\,|\theta^n|\ge 1$ on $I_n$, and $\theta^n=0$ on $I\setminus I_n$.
    We set 
    \(\lambda_n = \theta^n|u^n|^2\) and $p_n = u^n/|u^n|$,
    which is well defined since $|u^n|\ge|M^{\B, n}_{11}|\ge c$ on $I_n$.
\end{proof}

\subsection{Conclusion}\label{subsec:conclusion}
In this subsection we prove Theorem \ref{teo:perturbation}.

\begin{proof}[Proof of Theorem \ref{teo:perturbation}] 
Let $v^\pm(t)\in C^1(I;\R^3)$ be such that $v^\pm(t)\in V^\pm_t$ for every $t\in I$, and define \[\pi_{v^\pm(t)} = \text{span}\{e_2, v^\pm(t)\} =: \{(x, y, z)\in\R^3:\:\alpha^\pm(t) x+\beta^\pm(t) z = 0\}.\]
\textbf{First step.} We prove the existence of a sequence $(M^n)_n = (\lambda_n p_n\otimes p_n)_n$ with all the required properties, except the regularity.

In view of Lemma \ref{lemma:approx5} we may suppose that $M\in\mathcal{P}^\B(I;\rsym)$ and that $A_M\in\mathcal{P}(I;\tilde{\U}_c\cap\tilde{\U}^{(\alpha^+, \beta^+)}\cap\tilde{\U}^{(\alpha^-, \beta^-)})$ for some constant $c>0$. By Theorem \ref{teo:approx1} there exist $(\tilde{\lambda}_n)_n\subset L^\infty(I)$ and $(\tilde{p}_n)_n\subset L^\infty(I;\mathbb{S}^1)$ such that, setting $\tilde{M}^n:=\tilde{\lambda}_n\tilde{p}_n\otimes\tilde{p}_n$, it holds $\tilde{\lambda}_n=0$ on $I\setminus I_n,\,|\tilde{M}_{11}^{\B, n}|\ge c$ a.e. on $I_n,\,\tilde{M}^n\weakstar M$ in $L^\infty(I),$ and $\F(\tilde{M}^n)\to\F(M)$. We set $\tilde{A}^n \coloneqq A_{\tilde{M}^n}$ and $A \coloneqq A_M$. Note that 
\begin{equation}\label{eq:M}
	\tilde{M}^{\B, n}_{22} = \begin{cases}
		0&\text{ if }\tilde{A}^n_{13}=0,\\
		\frac{(\tilde{A}^n_{23})^2}{\tilde{A}^n_{13}}&\text{ if }\tilde{A}^n_{13}\neq 0.
	\end{cases}
\end{equation}

We now perturb the sequence $\tilde{A}^n$ in order to make it admissible. To do so, we follow the same strategy as in \cite[Lemma 3.7]{mora1}. Since $A = A_M$ we have that $\tilde{A}^n\weakstar A$ in $L^\infty(I;\tilde{\U}).$ Let $J\subseteq (1/4, 3/4) = I_4$ be an open interval on which $M$ is constant. We define 
\[\tilde{E}\coloneqq\{C\in \mathcal{P}(I;\U):\:C(x) = 0\text{ for a.e. }x\in I\setminus J\}.\]
    Since $A$ is nondegenerate on $J$, by Lemma \ref{lemma:perturbation_hornung} there is a finite dimensional subspace $E\subseteq\tilde{E}$ and $(\hat{A}^n)_n\in E$ converging to zero uniformly such that $A^n\coloneqq \tilde{A}^n+\hat{A}^n$ is admissible for every $n$. In particular, $A^n$ converges weakly$^*$ to $A$ in $L^\infty(I;\tilde{\U}).$ 
    
    Define 
    \[\gamma_n = \begin{cases}
    	(A^n_{23})/A^n_{13}&\text{ if }A^n_{13}\neq 0,\\
    	0&\text{ if }A^n_{13}=0.
    \end{cases}\]
    By construction $A^n = \tilde{A}^n$ in $I\setminus J$. Therefore equation \eqref{eq:M} implies that $\gamma_n = \tilde{M}_{22}^{\B, n}$ in $I\setminus I_4$. On the other hand, 
    \[|\tilde{A}^n_{13}| = |\tilde{M}_{11}^{\B, n}\ge c\text{ a.e. on }I_4.\]   
    
     Since $\hat{A}^n\to 0$ uniformly, we deduce that $A_{13}^n$ is also bounded away from zero on $I_4$. Therefore, using that $\tilde{A}^n$ and $A^n$ are uniformly bounded, we have
    \begin{align*}
        |\gamma_n-\tilde{M}^{\B, n}_{22}| &= \Bigl|\frac{(A^n_{23})^2}{A^n_{13}}-\frac{(\tilde{A}^n_{23})^2}{\tilde{A}^n_{13}}\Bigl| = \Bigl|\frac{(A^n_{23})^2\tilde{A}^n_{13}-(\tilde{A}^n_{23})^2A_{13}^n}{A^n_{13}\tilde{A}^n_{13}}\Bigl|\le\\
        &\le \frac{1}{c^2}\Bigl(|\tilde{A}^n_{13}||(A_{23}^n)^2-(\tilde{A}_{23}^n)^2|+|(\tilde{A}^n_{23})^2||\tilde{A}^n_{13}-A^n_{13}|\Bigl)\le\\
        &\le C|\hat{A}^n_{23}|+C|\hat{A}^n_{13}|,
    \end{align*}
    for some constants $c, C$. We conclude that 
    \begin{equation}\label{eq:25}
    	\gamma_n-\tilde{M}_{22}^{\B, n}\to 0\text{ uniformly on }I.
    \end{equation}

    We now set 
    \begin{equation}\label{eq:M_n}
        M^n\coloneqq A_{13}^nB'\otimes B'+A^n_{23}(B'\otimes N+N\otimes B')+\gamma_n N\otimes N.
    \end{equation}
    From \eqref{eq:25}, the convergence of $\F(\tilde{M}^n)$ to $\F(M)$, and the uniform convergence of $\hat{A}^n$, we deduce that $M^n\weakstar M$ and $\F(M^n)\to\F(M)$.

    It remains to show that $M^n$ is of the form $\lambda_np_n\otimes p_n$ for some $\lambda_n$ and $p_n$ satisfying the desired properties.
    We define $p_n = B'$ if $x\in I\setminus I_n$, and
    \begin{equation}\label{eq:p_n}
    p_n\coloneqq\alpha_n\bigl(A^n_{13}B'+A^n_{23}N\bigl)\;\;\text{ in } I_n,\;\text{ with }\alpha_n\coloneqq \frac{\sign(A^n_{13})}{\Bigl((A_{13}^n)^2+(A_{23}^n)^2\Bigl)^{1/2}}.
    \end{equation}
    We set $\lambda_n = 0$ on $I\setminus I_n$, and 
    \begin{equation}\label{eq:lambda_n}
        \lambda_n\coloneqq \frac{A^n_{13}}{(A^n_{13})^2(\alpha_n)^2} = \frac{(A^n_{13})^2+(A^n_{23})^2}{A^n_{13}}
    \end{equation}
    in $I_n$.
    It is easy to verify that $M^n = \lambda_np_n\otimes p_n$, and that $\lambda_n=0$ on $I\setminus I_n,\,p_n = B'$ on $I\setminus I_n$ and $p_n\cdot B'>0$ on $I$. Moreover, since $|A^n_{13}|\ge c/2$ a.e. on $I_n$ for $n$ large enough, it follows that 
    \[|\lambda_n| = \Bigl|A^n_{13}+\frac{(A^n_{23})^2}{A^n_{13}}\Bigl|\ge|A^n_{13}|\ge c/2,\text{ a.e. in }I_n\]
    for a suitable constant $c$. Moreover, since $\hat{A}^n_{13}\to 0$ uniformly, and $|\tilde{A}^n_{13}|\ge c$, on $I\setminus I_n$
    \[\sign(\lambda_n) = \sign(A^n_{13}) = \sign(\tilde{A}^n_{13}+\hat{A}^n_{13})= \sign(\tilde{A}^n_{13})\]
    for every $n$ sufficiently large. Thus, $\sign(\lambda_n)\in\mathcal{P}(I)$.\\
    \textbf{Second step.} We now prove that we can find $\lambda_n$ and $p_n$ with the desired regularity.
    We only sketch the proof since it follows the same strategy of \cite[Proposition 3.1]{mora1}.\\
    1) In view of the first step we can assume that $M = \lambda p\otimes p$, where $\lambda\in L^\infty(I)$ satisfies $\sign(\lambda)\in\mathcal{P}(I),\;\lambda(x) = 0$ for a.e. $x\in I_{\bar{n}}$ for some $\bar{n}$ and $|\lambda(x)|\ge c$ for a.e. $x\in I_{\bar{n}}$, while $p\in L^\infty(I;\mathbb{S}^1)$ satisfies $p = B'$ on $I_{\bar{n}}$ and $p\cdot B'>0$ a.e. on $I$, provided that we exhibit a sequence $(M^n)_n$ satisfying all the required properties with strong rather than weak convergence in Theorem \ref{teo:perturbation}-(iii).\\
    2) By the above mentioned properties of $p$ we can write \[p(t) = \cos\bigl(\theta(t)\bigl)B'(t)+\sin\bigl(\theta(t)\bigl)N(t),\]
    for a measurable function $\theta\in L^\infty\bigl(I;(-\pi/2, \pi/2)\bigl)$ such that $\theta(t)\equiv 0$ for a.e. $t\in I\setminus I_{\bar{n}}$ and $\theta_{|_{I_{\bar{n}}}}\in\mathcal{P}(I_{\bar{n}})$. Mollifying $\theta$ we find functions $\tilde{p}_n\in C^1(\bar{I};\mathbb{S}^1)$ converging to $p$ boundedly in measure (that is, $\tilde{p}_n$ converges in measure and is bounded with respect to $\norma{\cdot}_{\infty}$), and such that, after possibly renaming a bigger $\bar{n},$ the sequence satisfies $\tilde{p}_n(t) = B'(t)$ on $I\setminus I_{\bar{n}}$ and $\tilde{p}_n\cdot B'\ge \tilde{c}\text{ a.e. on }I_{\bar{n}}$ since the same property holds for $p$.\\
    3) Let $\tilde{J}\subseteq(1/4, 3/4)$ be an open interval on which $\sign(\lambda)$ is constant. Mollifying $\lambda$ we find a sequence $\tilde{\lambda}_n\in C^\infty(\bar{I})$ converging to $\lambda$ boundedly in measure and such that, after possibly relabeling the sequence, $(\tilde{\lambda}_n)_{|_{I\setminus I_n}} = 0$ and $|\tilde{\lambda}_n|\ge \tilde{c}$ on a suitable open set $J\subset\subset\tilde{J}$ for $n$ sufficiently large.\\
    4) We define $\tilde{M}^n\coloneqq\tilde{\lambda}_n \tilde{p}_n\otimes \tilde{p}_n\in C^1(\bar{I};\rsym)$ so that $A_{\tilde{M}^n}\to A_M$ boundedly in measure and $ \tilde{M}^{\B,n}_{22}\to M^\B_{22}$ in the same sense. Indeed, since $\tilde{p}_n\cdot B'$ and $p\cdot B'$ are bounded away from zero, we have
    \[|\tilde{M}^{\B, n}_{22}-M^\B_{22}| = \Bigl|\tilde{\lambda}_n|\tilde{p}_n\cdot N|^2 - \lambda_n|p\cdot N|^2\Bigl|,\]
    and we deduce the claim since $\tilde{\lambda}_n\to\lambda$ and $\tilde{p}_n\to p$ boundedly in measure.\\
    5) Since $A_M$ is admissible and nondegenerate on $J$, by Lemma \ref{lemma:perturbation_hornung} there exist a finite dimensional subspace $E\subseteq C_c^\infty(J;\U)$ and some $\hat{A}^n\in E$ converging to zero uniformly such that $A^n\coloneqq\hat{A}^n+A_{\tilde{M}^n}$ are admissible. We have that $A^n\in C^1(\bar{I};\tilde{\U})$ and $A^n\to A_M$ boundedly in measure. Moreover for every $x\in J$ it holds
    \[|\tilde{M}^{\B,n}_{11}| = |\tilde{\lambda}_n||\tilde{p}_n\cdot B'|^2\ge c,\]
    therefore $A^n_{13} = \tilde{M}^{\B,n}_{11}+\hat{A}^n_{13}$ is bounded away from zero on $J$ for $n$ large enough. Defining $M^n$ as in (\ref{eq:M_n}), and arguing as in the previous step, we conclude that $M^n\to M$ strongly in $L^2(I;\rsym).$\\
    6) We define $p_n(x) = \tilde{p}_n(x)$ for a.e. $x\in I\setminus J$, and otherwise we define it as in (\ref{eq:p_n}). Similarly, we set $\lambda_n(x) = \tilde{\lambda}_n(x)$ for a.e. $x\in I\setminus J$, and as in (\ref{eq:lambda_n}) if $x\in J$. Since $p_n$ and $\lambda_n$ are uniquely determined at every point $x$ such that $A^n_{13}(x)\neq 0,$ we have that $p_n = \tilde{p}_n$ and $\lambda_n = \tilde{\lambda}_n$ in an open neighborhood of $J$. Since in a neighborhood of $J$, $\sign (A^n_{13}) = \sign(\tilde{M}^{\B,n}_{11})$ for $n$ large, and $\sign(\tilde{M}^{\B,n}_{11})$is constant in a neighborhood of $J$, we conclude that $p_n\in C^1(\bar{I};\mathbb{S}^1)$ and $\lambda_n\in C^1(\bar{I}).$ This completes the proof.
\end{proof}

\section{The recovery sequence}\label{sec:recovery}
In this section we prove Theorem \ref{teo:gamma_convergence}. Since the proof of the liminf inequality is identical to that in \cite{mora2}, here we only prove Theorem \ref{teo:gamma_convergence}-(ii). We conclude the section with Remark \ref{rem:degenerate_case}, which addresses the case where $\A_0$ is not empty, while both $\bar{y} =  \chi(\ell, 0)$ and $\bar{R}^T =  (D_1(\ell)|D_2(\ell)|e_3)$ hold.

We precede the proof of the limsup inequality with the following Remark.

\begin{remark}\label{rem:degenerate_curves}
future references that if $\bigl(y, R^T\bigl)\in\A_0$ satisfies $d_1'\cdot d_3 = 0$ and $d_2'\cdot d_3 = 0$ a.e. in an open interval $J\subseteq I,$ and if $\bigl(y(\bar{x}), R^T(\bar{x})\bigl) = \bigl(\chi(\bar{x}, 0), (D_1(\bar{x})|D_2(\bar{x})|e_3)\bigl)$ for some $\bar{x}\in J$, then necessarily $y(\cdot) = \chi(\cdot, 0)$ and $(d_1|d_2|d_3) = (D_1|D_2|e_3)$ in $J$. In particular, if $J = I$, then $\bar{y} = \chi(\ell, 0)$ and $\bar{R}^T = (D_1(\ell)|D_2(\ell)|e_3)$.
	
	In fact since there exist two functions $\alpha_1(t),\,\alpha_2(t)$ such that $d_3\wedge d_1 = \alpha_1d_1+\alpha_2d_2$, we deduce that  
	\[d_3' = -(d_1'\cdot d_3)d_1-\bigl((d_3\wedge d_1)'\cdot d_3\bigl)(d_3\wedge d_1) = 0\;\text{ a.e. in }J.\]
	Now from $d_3(\bar{x}) = e_3$ it follows that $d_3\equiv e_3$ a.e. in $J$. Moreover from $d_1(\bar{x}) = D_1(\bar{x})$ and
	\[d_1'\cdot (e_3\wedge d_1) =  D_1'\cdot (e_3\wedge D_1),\]
	we deduce that $d_1 \equiv D_1$ in $J$ and $e_3\wedge d_1\equiv e_3\wedge D_1$ in $J$. Finally from (\ref{eq:limit_space}) and 
	\[d_2 = (d_2\cdot d_1)d_1+\bigl(d_2\cdot(e_3\wedge d_1)\bigl)e_3\wedge d_1\]
	it follows
	\[(d_2\cdot d_1)^2+\bigl(d_2\cdot(e_3\wedge d_1)\bigl)^2 = |d_2|^2 = |D_2|^2 = (D_1\cdot D_2)^2+\bigl(D_2\cdot(e_3\wedge D_1)\bigl)^2
	\]
	and since $d_2\cdot d_1 = D_2\cdot D_1$ we deduce that 
	\[\bigl(d_2\cdot(e_3\wedge d_1)\bigl)^2 = \bigl(D_2\cdot(e_3\wedge D_1)\bigl)^2\]
	a.e. in $J$. Since $d_3 = e_3 = \frac{d_1\wedge d_2}{|d_1\wedge d_2|} = \frac{D_1\wedge D_2}{|D_1\wedge D_2|}$ we conclude that $d_2 = D_2$ in $J$.
\end{remark}

\begin{proof}
    [Proof of Theorem \ref{teo:gamma_convergence}-(ii).]
    The proof closely follows the one presented in \cite{mora2}. Here we only sketch it, pointing out the differences. \\
    1) Given $(y, R) = \bigl(y, (d_1| d_2|d_3)^T\bigl)\in\A_0$ we set
        \[\mu\coloneqq d_1'\cdot d_3,\;\tau\coloneqq d_2'\cdot d_3,\;\text{ and }\;k\coloneqq d_1'\cdot (d_3\wedge d_1) = B'\cdot (e_3\wedge D_1).\]
        We define $D^\alpha\coloneqq D^{-T}e_\alpha$ so that $D^\alpha\cdot D_\beta = \delta_{\alpha\beta}$, and we set
        \begin{equation}\label{eq:MM}
        M\coloneqq\mu D^1\otimes D^1+\tau(D^1\otimes D^2+D^2\otimes D^1)+\gamma D^2\otimes D^2,
	    \end{equation}
        where $\gamma$ is chosen so that
    \[\bar{Q}(x_1, \mu, \tau) = \bigl(Q(M-D^{-T}\Pi^oD^{-1})+\alpha^+(\det M)^++\alpha^-(\det M)^-\bigl)\det D.\]
    As in the proof of \cite[Theorem 5-(ii)]{mora2} we check that $\gamma$ exists and belongs to $L^2(I).$\\
    2) We can prove that $R_{A_M} = (d_1|d_3\wedge d_1|d_3)^T$ with the same argument as in the proof of \cite[Theorem 5-(ii)]{mora2}, using that $(d_3\wedge d_1)'\cdot d_3 = A_{23} = MB'\cdot N$. In particular $A_M$ is admissible for the data $\bar{y}$ and $\bar{R}$. Moreover, since we assumed $(\bar{y}, \bar{R}^T)\neq (B(\ell), (D_1(\ell)|D_2(\ell)|e_3))$, by Remark \ref{rem:degenerate_curves} it follows that the map $A_M$ is nondegenerate.
    
    Thus, by Theorem \ref{teo:perturbation} there are $\lambda_n\in C^1(\bar{I}),\;p_n\in C^1(\bar{I};\mathbb{S}^1)$ such that setting $M^n \coloneqq \lambda_np_n\otimes p_n\in C^1(\bar{I};\rsym)$ the following hold:
            \begin{align}
                &p_n\cdot B'>0,\;\text{ everywhere on }\bar{I},\;p_n = B'\text{ near }\partial I, \,\text{ and }\lambda_n = 0 \text{ near }\partial I\text{ for every }n,\label{eq:punto_1}\\
                &A_{M^n}\text{ is admissible for the data }\y,\,\RR,\label{eq:punto_2}\\
                &M^n\rightharpoonup M\text{ weakly in }L^2(I;\rsym),\label{punto_3}\\
                &\int_IQ(M^n)\det(D)\dx = \F(M^n)\to \F(M) = \int_I\Bigl(Q(M)+\alpha^+(\det M)^++\alpha^-(\det M)^-\Bigl)\det(D)\dx.\label{punto_4}
            \end{align}
            From (\ref{punto_4}) it follows
            \begin{equation}\begin{aligned}\label{eq:conv_G_J}
                \int_I\bigl(Q(M^n)-&\K D^{-T}\Pi^oD^{-1}\cdot M^n+Q(D^{-T}\Pi^oD^{-1})\bigl)\det(D)\dx \to \\
                &\to\int_I\Bigl(Q(M-D^{-T}\Pi^oD^{-1})+\alpha^+(\det M)^++\alpha^-(\det M)^-\Bigl)(\det D)\dx =  J(y, R).
            \end{aligned}
            \end{equation}

            Since $\lambda_n$ and $p_n$ are regular up to the boundary we can extend them in an open interval $J$ containing $\bar{I}$. For instance we take $\lambda_n\equiv 0$ on $J\setminus I$ and $p_n(t) = \partial_1\chi(t, 0) = B'(t)$ in $J\setminus I$. Correspondingly, we can assume that $M^n$ is defined on $J$. In this way $\lambda_n \in C^1(\bar{J})$ and $p_n\in C^1(\bar{J};\mathbb{S}^1)$, therefore by \cite[Lemma 12]{mora2} for every $n$ there is $\eta_n>0$ such that the map
            \[\Phi^n:J\times (-\eta_n, \eta_n)\to\R^2,\;\;\Phi^n(t, s) = B(t)+sp_n^\perp(t)\]
        is a bi-Lipschitz homeomorphism onto the open set $U^n\coloneqq \Phi^n\bigl(J\times(-\eta_n, \eta_n)\bigl)$.

    3) We let $R^n\coloneqq R_{A_{M^n}}$, which is defined on $J$ for every $n$, for $i=1, 2, 3$ we set $d_i^n\coloneqq (R^n)^Te_i$, and for $t\in J$ we define 
    \[\beta^n(t)\coloneqq \int_0^td^n_1(s)\ds.\]
    Then, by \cite[Proposition 13]{mora2} for each $n$ there exist an isometry $u^n:U^n\to\R^3$ such that $u^n\circ B = \beta^n$ and $\Pi_{u^n}\circ B = M^n$ on $J$, and \begin{equation}\label{eq:grad_u}
        (\nabla u^n)\circ \bigl(B(t)+sp^\perp(t)\bigl)= (R^n(t))^Te_1\otimes B'(t)+(R^n(t))^Te_2\otimes N(t)\quad \text{on }U^n.
    \end{equation}
    4) Since for $\e$ sufficiently small $S_\e = \chi_\e(\Omega)\subseteq U^n$ we can define
    \[y^n_\e\coloneqq u^n\circ\chi_\e.\]
    Following the same strategy used in the proof in \cite{mora2}, we can check that there is a subsequence $y^n_{\e_n}$ such that
    \[y^n_{\e_n}\rightharpoonup y\text{ weakly in }W^{2, 2}(\Omega;\R^3)\text{ and }\nabla_{\e_n}y^n_{\e_n}\rightharpoonup (d_1|d_2)\text{ weakly in }W^{1, 2}(\Omega;\R^3),\]
    and 
    \[\limsup_{\e_n\to 0}J_{\e_n}(y^n_{\e_n})\le J(y, R).\]
    5) To conclude, we need to prove that $y^n_{\e_n}\in\A_{\e_n}^\Omega$ for all $n$ sufficiently large.

    We observe that for any $\e$ it holds \begin{equation*}
        \begin{aligned}
        y^n_{\e}(0, 0) &= u^n\circ\chi_{\e}(0, 0) = u^n\circ B(0) = \beta^n(0)= 0,\\
        y^n_{\e}(\ell, 0) &= u^n\circ\chi_{\e}(\ell, 0) = u^n\circ B(\ell) = \beta^n(\ell)= \int_0^\ell d^n_1(s)\ds = \y,
        \end{aligned}
        \end{equation*}
    where last equality holds since $A_{M^n}$ is admissible.

    Moreover by (\ref{eq:punto_1}) we have $M^n(t) = 0$ in a neighborhood of $\partial I$ contained in $J$, and, in particular, $A_{M^n} = k(e_1\otimes e_2-e_2\otimes e_1)$ near $\partial I$. Since by definition we have $k = B''\cdot N$, we deduce that also $(B'(t)|N(t)|e_3)$ is a solution of the Cauchy problem \eqref{eq:cauchy_problem} near $\partial I$, with $A_{M^n} = k(e_1\otimes e_2-e_2\otimes e_1)$ in place of $A$. As a consequence there exists $\delta_n>0$ such that $(R^n(t))^T = (B'(t)|N(t)|e_3)$ for every $t\in(-\delta_n, \delta_n)\subseteq J$. By \eqref{eq:grad_u} $(\nabla u^n)\circ \Phi^n(t, s) = (e_1|e_2)$ for every $(t, s)\in (-\delta_n, \delta_n)\times(-\eta_n, \eta_n),$ which is an open neighborhood of $(0, 0)$. Therefore for $\e$ sufficiently small \[\chi_\e\Bigl(\{0\}\times(-1/2, 1/2)\Bigl)\subset\Phi_n\Bigl((-\delta_n, \delta_n)\times(-\eta_n, \eta_n)\Bigl),\]and, in particular, we deduce $(\nabla u^n)\circ\chi_\e(0, x_2) = (e_1|e_2)$ for every $x_2\in(-1/2, 1/2)$. It follows that we can choose the sequence $\e_n$ so that 
    \[(\nabla y^n_{\e_n})(0, x_2) = (\nabla u^n)\circ\chi_{\e_n}(0, x_2)\cdot\nabla\chi_{\e_n}(0, x_2) = \nabla\chi_{\e_n}(0, x_2)\]
    for every $x_2\in (-1/2, 1/2)$, which proves the second condition of (\ref{rescaled_boundary_cond_2}). To prove the third condition we observe that
    the map 
    \[\tilde{R}^T\coloneqq \bar{R}^T
    \begin{pmatrix}
        B'|N|e_3
    \end{pmatrix}
    \]
    satisfies $\tilde{R}' = \tilde{A}\tilde{R}$ where
    \[\tilde{A} = \begin{pmatrix}
        0&k&0\\-k&0&0\\0&0&0
    \end{pmatrix},
    \]
    with boundary datum $\tilde{R}^T(\ell) = \bar{R}^T\bigl(B'(\ell)|N(\ell)|e_3\bigl).$
    Since $\tilde{A} = A_{M^n}$ in a neighborhood of $\ell$ and since $\tilde{R}_n(1) = \tilde{R}(1)$, then for every $n$ there is $\delta_n>0$ such that $\tilde{R}(t) = R^n(t)$ for every $t\in(\ell-\delta_n, \ell+\delta_n).$ We can conclude the proof as before, observing that then
    $(\nabla u^n)\circ\Phi^n(t, s) = (\bar{d}_1|\bar{d}_2)$ for every $(t, s)\in(-\delta_n, \delta_n)\times(-\eta_n, \eta_n).$
\end{proof}

\begin{remark}\label{rem:degenerate_case}
	In this remark, we address the case where $\mathcal{A}_0$ is non-empty, while both 
	\begin{equation}\label{eq:last}
		\bar{y} = \chi(\ell, 0) \quad \text{and} \quad \bar{R}^T = (D_1(\ell)|D_2(\ell)|e_3)
	\end{equation}
	hold. 
	
	First of all, we observe that the liminf inequality of Theorem \ref{teo:gamma_convergence} holds true with no changes, so we only focus on the limsup inequality. 
	Let the map $M$ be defined as in \eqref{eq:MM}. Following Step 2) of the proof of Theorem \ref{teo:gamma_convergence}, the matrix $A_M$ is admissible for the boundary data $(\bar{y}, \bar{R})$. If $A_M$ is non-degenerate, the proof of the limsup inequality proceeds as previously established.
	
	Suppose instead that $A_M$ is degenerate. From Remark \ref{rem:degenerate_curves}, we deduce that $(y, R^T) = (B, (D_1|D_2|e_3))$. We distinguish two cases based on whether $\Pi_0$ vanishes.
	
	First, assume $\Pi_0 = 0$ on $I$. It follows from \eqref{def:limit_functional} that $M = 0$ and $J(y, R^T) = 0$. Interpreting $M$ as the second fundamental form of all the isometries of the recovery sequence, we define $u_\varepsilon(x, y) = (x, y, 0)$, which yields the recovery sequence $y_\varepsilon := \chi_\varepsilon$ for all $\varepsilon$. In particular, in this instance the limit functional is consistently given by \eqref{def:limit_functional}.
	
	Conversely, suppose $\Pi_0$ is not identically zero. In this case, explicitly providing a recovery sequence remains an open problem, the main geometric obstruction is as follows. For this state, the limit energy is minimized by the ``second fundamental form'' $M = \gamma D^2 \otimes D^2$, where $\gamma \in L^2(I)$ is not necessarily vanishing. Theorem \ref{teo:perturbation} cannot be applied directly due to the degeneracy of $A_M$. 
	
	A natural strategy would be to construct a sequence of isometries such that $M$ serves as the exact second fundamental form along the centerline $B$, mirroring the approach used for the case $\Pi_0 = 0$. However, since any smooth isometric immersion of a flat domain must be a ruled surface, this construction requires extending the centerline $B$ along straight rulings whose projections onto the $xy$-plane are oriented along the eigenvector corresponding to the zero eigenvalue of $M$.
	
	In our setting, $M = \gamma D^2 \otimes D^2$, which implies that this eigenvector is $D_1$. Since $D_1 = B'$, the projections of these rulings are perfectly tangent to the centerline, a configuration that prevents the construction of a valid sequence of non-singular isometries. In particular, \cite[Proposition 13]{mora2} is inapplicable for constructing an admissible isometry. Indeed that proposition requires the transversality condition $p \cdot B' \neq 0$, which is violated in our context because $p$ corresponds to $D^2$ and, as noted, $D^2 \cdot B' = 0$.
\end{remark}

\section*{Acknowledgments}
I would like to express my gratitude to my supervisor, Prof. Maria Giovanna Mora, for her guidance and for her careful reading of the manuscript drafts.

\nocite{*} 

\printbibliography

\vspace{1cm}
\noindent
(Giovanni Savar\'e)
\small Technische Universit\"at M\"unchen, Boltzmannstrasse 3, 85748 Garching, Germany\\
\small \texttt{giovanni.savare@tum.de}

\end{document}